\documentclass[11pt]{amsart}

\usepackage[utf8]{inputenc}
\usepackage[T1]{fontenc}
\usepackage[english]{babel}
\usepackage{amsmath,amsthm,amsfonts,amssymb,mathtools}
\usepackage{enumitem}
\usepackage{geometry}
\usepackage{hyperref}
\usepackage{microtype}

\hypersetup{
  colorlinks=true,
  linkcolor=blue,
  citecolor=blue,
  urlcolor=blue,
  pdftitle={Splitting fields of complex hyperbolic groups: matrix extraction, Hermitian descent, and spectral fields},
  pdfauthor={Angel Cano, Hector Castro, Nikolay Gusevskii},
  pdfkeywords={complex hyperbolic geometry, trace fields, fields of definition, unitary groups, Hermitian forms, Vandermonde systems, splitting fields, Schottky groups}
}

\newtheorem{theorem}{Theorem}[section]
\newtheorem{proposition}[theorem]{Proposition}
\newtheorem{lemma}[theorem]{Lemma}
\newtheorem{corollary}[theorem]{Corollary}
\newtheorem{conjecture}[theorem]{Conjecture}
\newtheorem{problem}[theorem]{Problem}
\theoremstyle{definition}
\newtheorem{definition}[theorem]{Definition}
\newtheorem{example}[theorem]{Example}
\theoremstyle{remark}
\newtheorem{remark}[theorem]{Remark}

\newcommand{\C}{\mathbb C}
\newcommand{\R}{\mathbb R}
\newcommand{\Q}{\mathbb Q}

\newcommand{\M}{\mathrm M}
\newcommand{\GL}{\mathrm{GL}}
\newcommand{\SL}{\mathrm{SL}}
\newcommand{\SU}{\mathrm{SU}}
\newcommand{\SO}{\mathrm{SO}}
\newcommand{\PU}{\mathrm{PU}}

\newcommand{\tr}{\operatorname{tr}}

\newcommand{\Span}{\operatorname{Span}}

\newcommand{\Gal}{\operatorname{Gal}}

\newcommand{\diag}{\operatorname{diag}}
\newcommand{\Ad}{\operatorname{Ad}}
\newcommand{\disc}{\operatorname{disc}}
\newcommand{\Tr}{\operatorname{Tr}}

\begin{document}

\title[Splitting fields of complex hyperbolic groups]{Splitting fields of complex hyperbolic groups:\\ matrix extraction, Hermitian descent, and spectral fields}

\author{\'Angel Cano}
\address{Instituto de Matem\'aticas, Universidad Nacional Aut\'onoma de M\'exico, Unidad Cuernavaca, Av. Universidad s/n, Col. Lomas de Chamilpa, C.P. 62210, Cuernavaca, Morelos, Mexico}
\email{angelcano@im.unam.mx}

\author{Hector Castro}
\address{Instituto de Matem\'aticas, UNAM, Unidad Cuernavaca, Av. Universidad s/n, Col. Lomas de Chamilpa, C.P. 62210, Cuernavaca, Morelos, Mexico}
\email{abraham.n.castro@gmail.com}

\author{Nikolay Gusevskii}
\address{Departamento de Matem\'atica, Instituto de Ci\^encias Exatas, Universidade Federal de Minas Gerais, Av. Ant\^onio Carlos 6627, Pampulha, Belo Horizonte, MG 31270-901, Brazil}


\begin{abstract}
Let $\Gamma<\SU(n,1)$, $n\ge2$, be Zariski-dense and contain a regular loxodromic element $A$.  The splitting-field problem asks how far traces and spectra determine a field of definition for $\Gamma$.  We show first that traces together with the spectrum of $A$ always give a linear realization over $K=\Q(\tr\Gamma,E(A))$.  We also prove that every complex-irreducible subgroup of $\SU(n,1)$ with real trace field is conjugate into $\SO(n,1)$; in particular, no discreteness hypothesis is required for this real-trace conclusion.  In higher dimension, however, linear descent does not by itself give the standard unitary form: a Hermitian similarity obstruction remains, described over imaginary quadratic fields by Landherr's local invariants.  In $\SU(3,1)$ we give a Zariski-dense example for which the one-spectrum field is $\Q(i)$ but does not yield the standard unitary form.  This leads naturally to the total spectral field.  We prove that the obstruction disappears for arithmetic groups and under suitable local openness hypotheses, and we give an even-degree local criterion.  In dimension four we analyze the discrete branch through explicit Schottky groups and rank-one local geometry.  There the full Bruhat cocycle recovers anisotropic Levi holonomy and forces a non-split rotational word.  The remaining quasi-split rank-two case is not reached by this method, so the general $\SU(3,1)$ problem remains open.
\end{abstract}

\maketitle

\section{The splitting-field problem}

A recurring arithmetic question in hyperbolic geometry is how much of a group can be recovered over a field generated by its traces and eigenvalues.  In real hyperbolic geometry the total eigenvalue field already plays this role in the work of Hamilton--Reid \cite{HamiltonReid2004}.  For complex hyperbolic groups, Cunha--Gusevskii proved a field-of-definition theorem in $\SU(2,1)$, and the higher-dimensional splitting-field problem was pursued in \cite{Batista2018,Genzmer2014,Gusevskii2018,CunhaGusevskii2014}.  The present paper continues that line of questions.

The difficulty that appears in higher dimension is simple to state.  Recovering matrix coefficients over a field $K$ and recovering the standard unitary form over the same field are different problems.  Thus there are two descent questions:
\begin{enumerate}[label=(\roman*)]
\item \emph{linear descent}: find a basis in which the matrices of $\Gamma$ have coefficients in $K$;
\item \emph{unitary descent}: realize the same group over $K$ while putting its invariant Hermitian form in the standard complex-hyperbolic similarity class.
\end{enumerate}
The first is controlled by traces and the spectrum of a regular loxodromic element.  The second is a Hermitian similarity problem.  In dimension four Theorem~\ref{thm:one-A-obstruction} shows that the field attached to one regular spectrum need not remove this additional obstruction.  Once this distinction is made, adjoining further spectra becomes a natural way to remove the remaining arithmetic obstruction.  Throughout the general $\SU(n,1)$ discussion we assume $n\ge2$, or equivalently $d=n+1\ge3$.

We formulate the main program in the Zariski-dense regime, which is the hypothesis used in the descent and local arguments below.  Earlier formulations in terms of total irreducibility remain part of the motivation, but we keep total irreducibility and Zariski density separate.

This brings us to the question that guides the paper.

\begin{problem}\label{prob:CG}
Let $\Gamma<\SU(n,1)$ be Zariski-dense and contain a regular loxodromic element $A$.  Set
\[
 K_A=\Q\bigl(\tr\Gamma,E(A)\bigr).
\]
To what extent is $K_A$ already a unitary field of definition?  If an additional extension is needed, can it be generated intrinsically from traces or splitting fields of further elements of $\Gamma$?  In particular, can the dependence on the chosen element $A$ be eliminated at the level of the final field of definition?
\end{problem}

Our first result gives a complete answer to the linear part.

\begin{theorem}\label{thm:intro-linear}
Let $\Gamma<\SU(J)$ be Zariski-dense and contain a regular loxodromic element
\[
 A=\diag(\lambda_1,\ldots,\lambda_d)
\]
with pairwise distinct eigenvalues.  Put
\[
 K=\Q(\tr\Gamma,\lambda_1,\ldots,\lambda_d).
\]
Then there exists $D\in\GL_d(\C)$, diagonal and commuting with $A$, such that
\[
 D^{-1}\Gamma D\subset\SL_d(K).
\]
\end{theorem}

The unitary question has an additional invariant.

\begin{theorem}\label{thm:intro-hermitian}
Assume $n\ge2$ and, in addition, that complex conjugation acts nontrivially on $K$, and put $K_0=K^{\bar{\ }}$; equivalently, $K/K_0$ is quadratic.  Then the $K$-rational representation of Theorem~\ref{thm:intro-linear} preserves a non-degenerate Hermitian form $H_K$ defined over $K$.  The group is realizable over the standard unitary form $J$ over $K$ if and only if $H_K$ is similar over $K$ to $J$.
\end{theorem}

The two theorems isolate the arithmetic issue.  Over imaginary quadratic fields, Landherr's classification turns the difference between $H_K$ and $J$ into finitely many local norm obstructions; the Gaussian example in Section~\ref{sec:hermitian} shows that the spectrum of a fixed loxodromic element need not remove them.  This leads to the total spectral field
\[
 F_{\rm eig}(\Gamma)=K\bigl(E(\gamma):\gamma\in\Gamma\bigr).
\]
For arithmetic groups, and more generally under suitable local openness hypotheses, this larger field removes the obstruction; in dimension four it is enough to force even local splitting degree at each obstructive place.

The remaining difficulty is local and discrete.  In the non-quasi-split rank-one case a distinguished transverse Hermitian plane gives a quadratic rotational factor.  The full Bruhat cocycle sees its anisotropic Levi holonomy, and an asymptotic composition argument produces a loxodromic word with non-split transverse spectrum.  The quasi-split rank-two case has no such distinguished quadratic factor and is not reached by this mechanism.

\begin{conjecture}\label{conj:rank2}
Let $F/F_0$ be a quadratic extension of non-archimedean local fields of characteristic zero and let $V$ be a four-dimensional \emph{quasi-split} Hermitian space over $F/F_0$.  Let
\[
 \Lambda<\SU(V)(F_0)
\]
be discrete, torsion-free, and Zariski-dense.  Then $\Lambda$ contains a regular semisimple element $g$ whose splitting field over $F$ has even degree.
\end{conjecture}

The corresponding global statement is therefore still open.

\begin{conjecture}\label{conj:total-main}
Let $K/\Q$ be imaginary quadratic and let $H_K$ be a Hermitian form of signature $(3,1)$.  If
\[
 \Gamma<\SU(H_K,K)
\]
is finitely generated and Zariski-dense, then the total eigenvalue field
\[
 F_{\rm eig}(\Gamma)=K\bigl(E(\gamma):\gamma\in\Gamma\bigr)
\]
is a unitary field of definition for $\Gamma$.
\end{conjecture}

Theorem~\ref{thm:partial-total} proves this conjecture whenever every obstructive place is either non-discrete or, in the discrete case, non-quasi-split.  Thus the arithmetic and rank-one branches are accessible by the methods developed here, while the quasi-split rank-two branch requires a different local realization principle.

Earlier formulations in Batista's thesis \cite{Batista2018} and in Gusevskii's 2018 abstract \cite{Gusevskii2018} motivated the higher-dimensional splitting-field problem.  Here we separate coefficient recovery from Hermitian similarity and use the total spectral field as the natural target of the program.  Theorem~\ref{thm:one-A-obstruction} shows that, in dimension four, the Hermitian-similarity step can be nontrivial.

The organization follows the mathematical reduction.  Section~\ref{sec:preliminaries} fixes the geometric terminology used throughout.  Sections~\ref{sec:original} and~\ref{sec:linear} pass from explicit coefficient recovery to invariant linear descent.  Section~\ref{sec:hermitian} identifies the Hermitian obstruction, and Sections~\ref{sec:total}--\ref{sec:S4} develop arithmetic and local ways of removing it.  Sections~\ref{sec:rank1}--\ref{sec:reset} analyze the discrete rank-one case, and Section~\ref{sec:assembly} assembles the local information globally.  The last sections compare lattice examples, return to the relation between traces and spectra, and isolate the remaining rank-two problem.

\section{Complex hyperbolic preliminaries}\label{sec:preliminaries}

We recall only the geometric facts that will be used later.  Our conventions follow the
standard projective model of complex hyperbolic space; compare, for example,
\cite{KimKim2016}.  We deliberately keep this section short and restrict it to
the notation needed in the sequel.

\subsection{The projective model}

Let $V=\C^{n+1}$ be equipped with a non-degenerate Hermitian form $h$ of
signature $(n,1)$.  After choosing an $h$-orthogonal basis one may write
\[
 h(z,w)=z_1\overline{w_1}+\cdots+z_n\overline{w_n}
          -z_{n+1}\overline{w_{n+1}}.
\]
Set
\[
 V_- = \{z\in V\setminus\{0\}:h(z,z)<0\},\qquad
 V_0 = \{z\in V\setminus\{0\}:h(z,z)=0\}.
\]
If $\pi:V\setminus\{0\}\to\mathbb P(V)$ denotes projectivization, then
\[
 \mathbb H^n_{\C}=\pi(V_-),\qquad
 \partial\mathbb H^n_{\C}=\pi(V_0).
\]
For $x=\pi(z)$ and $y=\pi(w)$ in $\mathbb H^n_{\C}$, the Bergman metric may
be normalized by
\[
 \cosh^2\!\left(\frac{d(x,y)}{2}\right)
 =\frac{|h(z,w)|^2}{h(z,z)h(w,w)}.
\]
This expression is independent of the chosen lifts $z$ and $w$.

The group $\mathrm U(h)$ preserves $V_-$ and therefore acts by holomorphic
isometries on $\mathbb H^n_{\C}$.  Its projectivization is the full group of
holomorphic isometries, while $\SU(h)$ is the determinant-one subgroup.  We
freely pass between the standard diagonal form above and the Siegel form $J$
used in the coefficient calculations below, since they are equivalent over
$\C$.

\subsection{Isometries and loxodromic elements}

An isometry of $\mathbb H^n_{\C}$ is called \emph{elliptic}, \emph{parabolic}, or
\emph{loxodromic} according as it has a fixed point in $\mathbb H^n_{\C}$,
exactly one fixed point on $\partial\mathbb H^n_{\C}$, or exactly two fixed
points on $\partial\mathbb H^n_{\C}$, respectively.  A loxodromic element of
$\SU(n,1)$ has two eigenvalues off the unit circle, of reciprocal moduli, while
all remaining eigenvalues lie on the unit circle.  We call a loxodromic element
\emph{regular} when its eigenvalues in the standard representation are pairwise
distinct.  This is the class of elements used throughout the spectral and
Vandermonde arguments.

For a subgroup $G<\SU(n,1)$, its limit set is the accumulation set on
$\partial\mathbb H^n_{\C}$ of an orbit in $\mathbb H^n_{\C}$.  A subgroup is
\emph{elementary} when its limit set is finite and \emph{nonelementary}
otherwise.  Discreteness will be imposed only where it is genuinely needed;
in particular, the real-trace theorem proved below does not require it.

\subsection{Totally geodesic subspaces and irreducibility}

The totally geodesic subspaces relevant here are of two types.  Complex
geodesic subspaces arise by projectivizing complex subspaces of $V$ on which
$h$ has signature $(k,1)$ and are isometric to $\mathbb H^k_{\C}$.  Totally
real geodesic subspaces arise from real subspaces on which $h$ takes real
values and are isometric to $\mathbb H^k_{\R}$.  In particular, after
conjugation the standard totally real copy $\mathbb H^n_{\R}\subset
\mathbb H^n_{\C}$ has stabilizer containing
\[
 \SO(n,1)<\SU(n,1).
\]
This elementary observation explains why complex irreducibility and Zariski
density are not equivalent: a complex-irreducible subgroup may be contained
in the proper real algebraic subgroup $\SO(n,1)$.

We say that $G<\SU(n,1)$ is \emph{complex irreducible} if it preserves no
proper nonzero complex linear subspace of $V$.  We use \emph{totally
irreducible} in the sense of the higher-dimensional splitting-field
literature: the group is complex irreducible and does not preserve a proper
totally real geodesic subspace.  These notions will be kept distinct from
Zariski density throughout the paper.

\section{Matrix coefficients from traces and spectra}\label{sec:original}

Fix a regular loxodromic element $A$.  The traces of words involving $A$ encode more matrix data than is visible at first sight.  Successive Vandermonde systems recover diagonal entries, pair products, mixed unitary products, and eventually individual coefficients.  This explicit calculation shows concretely where the field $\Q(\tr\Gamma,E(A))$ enters and points to the invariant trace-algebra argument of Section~\ref{sec:linear}.  The Hermitian similarity problem appears only after this linear step and is treated in Section~\ref{sec:hermitian}.

Let
\[
 d=n+1,
 \qquad
 J=
 \begin{pmatrix}
 0&0&1\\
 0&I_{d-2}&0\\
 1&0&0
 \end{pmatrix},
\]
and
\[
 \SU(J)=\{g\in\SL_d(\C):g^*Jg=J\}.
\]

\subsection{Zariski density and regular loxodromic elements}

The notions of complex irreducibility, total irreducibility, and Zariski density are recalled in Section~\ref{sec:preliminaries}.  The coefficient extraction below uses only Zariski density and one regular loxodromic element with simple spectrum.  We therefore keep the geometric irreducibility questions separate from the algebraic descent argument.  The relevant geometric background includes the Karpelevich--Mostow theorem; see \cite{DiScalaOlmos2009}.

\begin{proposition}\label{prop:two-generator}
Let $\Gamma<\SU(n,1)$ be Zariski-dense.  Then $\Gamma$ contains a non-abelian free subgroup generated by two elements whose Zariski closure is the ambient algebraic group.  In the coefficient-recovery arguments below we additionally choose, independently, a regular loxodromic element $A\in\Gamma$ with simple spectrum.
\end{proposition}

\begin{proof}
The first assertion is the strengthened free-subgroup theorem of Tits for Zariski-dense linear groups \cite{Tits1972}.  The regular loxodromic element $A$ is chosen independently.
\end{proof}

The two-generator subgroup is used only for the explicit coefficient calculations.  All later descent statements are made directly under Zariski density.  Once the trace algebra is split, the linear descent applies to all of $\Gamma$.

\begin{definition}
If $\Gamma<\SU(n,1)$ contains a loxodromic element $A$, we write
\[
 \Q(\Gamma,E(A))
 =\Q(\tr\Gamma,E(A))
\]
for the field generated by the trace field of $\Gamma$ and all eigenvalues of $A$.
\end{definition}

\subsection{The smaller trace--dilation field \texorpdfstring{$\Q(\tr\Gamma,L)$}{Q(tr Gamma,L)}}\label{subsec:trace-dilation-field}

The full spectrum of $A$ is not always needed to recover some of its arithmetic data.  Write
\[
 k(\Gamma)=\Q(\tr\Gamma).
\]
If the two eigenvalues of a loxodromic $A$ off the unit circle are
\[
 \alpha=Lu,\qquad \beta=L^{-1}u,
 \qquad L>1,\quad |u|=1,
\]
set
\[
 K_L=k(\Gamma,L).
\]
In $\SU(2,1)$ this field already contains the entire spectrum, which is the mechanism behind the field-of-definition theorem of Cunha--Gusevskii.  In higher dimension the analogous spectral statement fails.

\begin{proposition}[The trace--dilation field in low and higher dimension]\label{prop:trace-dilation-failure}
Let $G=\langle A\rangle$.  If $A\in\SU(2,1)$ is loxodromic with spectrum
\[
 Le^{i\varphi},\qquad e^{-2i\varphi},\qquad L^{-1}e^{i\varphi},
 \qquad L>1,
\]
then every eigenvalue of $A$ belongs to
\[
 \Q(G,L)=\Q(\tr G,L).
\]
For $n\ge3$, the corresponding assertion for a regular loxodromic element of $\SU(n,1)$ is false in general.
\end{proposition}

\begin{proof}
The $\SU(2,1)$ statement is Lemma~2.1 and Corollary~2.1 of Cunha--Gusevskii \cite{CunhaGusevskii2014}: their calculation gives $e^{i\varphi}\in\Q(G,L)$ from $\tr(A)$, $\overline{\tr(A)}$, and $L$, and hence all three eigenvalues lie in the same field.

For the higher-dimensional failure, consider in the Siegel model
\[
 A=\diag(2,i,-i,1/2)\in\SU(3,1).
\]
It is regular loxodromic and, for every $m\in\mathbb Z$,
\[
 \tr(A^m)=2^m+2^{-m}+i^m+(-i)^m\in\Q.
\]
Consequently
\[
 \Q(\tr\langle A\rangle,2)=\Q,
\]
whereas $i,-i\notin\Q$.  Thus adjoining only the real dilation parameter does not, in general, split the full characteristic polynomial in higher dimension.
\end{proof}

\begin{lemma}[Recovery of the two off-circle eigenvalues]\label{lem:extreme-eigenvalues}
Let $A\in\SU(n,1)$ be loxodromic with simple spectrum and let
\[
 \alpha=Lu,\qquad \beta=L^{-1}u,
 \qquad L>1,
\]
be its two eigenvalues off the unit circle.  If $A\in\Gamma$, then
\[
 \alpha,\beta\in K_L=\Q(\tr\Gamma,L).
\]
\end{lemma}

\begin{proof}
Let
\[
 P_A(X)=\det(XI-A).
\]
By Newton's identities, every coefficient of $P_A$ is a polynomial with rational coefficients in
\[
 \tr(A),\tr(A^2),\ldots,\tr(A^d),
 \qquad d=n+1.
\]
Since $A^m\in\Gamma$, these traces lie in $k(\Gamma)$, and therefore
\[
 P_A(X)\in k(\Gamma)[X]\subset K_L[X].
\]
Now $\alpha=L^2\beta$.  Hence $\beta$ is a common zero of $P_A(X)$ and $P_A(L^2X)$.  It is the only common zero: the unit-modulus eigenvalues cannot occur in a pair $(z,L^2z)$, while $\alpha$ cannot be the smaller member of such a pair.  Since the spectrum is simple,
\[
 \gcd\bigl(P_A(X),P_A(L^2X)\bigr)=X-\beta
 \quad\text{in }K_L[X].
\]
Thus $\beta\in K_L$, and then $\alpha=L^2\beta\in K_L$.
\end{proof}

\begin{remark}\label{rem:trace-dilation-limit}
The preceding results identify precisely what survives from the $\SU(2,1)$ trace--dilation mechanism.  In arbitrary dimension the two eigenvalues off the unit circle lie in $\Q(\tr\Gamma,L)$, but the remaining unit-circle eigenvalues need not.  Therefore the linear descent theorem below is formulated over the full one-spectrum field $\Q(\tr\Gamma,E(A))$.  Even if a smaller linear field is available in a particular case, unitary descent still requires the independent Hermitian-similarity condition studied in Section~\ref{sec:hermitian}.
\end{remark}

\begin{remark}[Complex irreducibility versus Zariski density]\label{rem:irr-zar}
Complex irreducibility and Zariski density will be kept separate.  Zariski density implies absolute irreducibility, but the converse is false in general: a complex-irreducible subgroup may lie in a proper real algebraic subgroup of $\SU(n,1)$, for example in the real form $\SO(n,1)$.  This is exactly the distinction used in the next two propositions.
\end{remark}

\begin{proposition}[Proposition 1 --- regular loxodromics in the Zariski-dense case]\label{prop:regular-lox-zariski}
Let $G<\SU(n,1)$, with $n\ge2$, be Zariski-dense.  Then $G$ contains a
loxodromic element whose eigenvalues in the standard representation are
pairwise distinct.
\end{proposition}

\begin{proof}
Let $\mathbf H=\SU(n,1)$, viewed as a connected semisimple real algebraic
group.  By Prasad--Rapinchuk \cite[Theorem~1]{PrasadRapinchuk2003}, the
Zariski-dense subgroup $G<\mathbf H(\R)$ contains an element $x$ which is
regular and $\R$-regular and for which
\[
 T=Z_{\mathbf H}(x)^\circ
\]
is a maximal real torus such that the cyclic subgroup $\langle x\rangle$ is
Zariski-dense in $T$.  In particular, $x$ is semisimple.

Let $T_s$ and $T_a$ denote respectively the maximal $\R$-split and
$\R$-anisotropic subtori of $T$.  By
\cite[Remark~1]{PrasadRapinchuk2003}, $T_s$ is a maximal $\R$-split torus of
$\mathbf H$.  Since $\mathbf H=\SU(n,1)$ has real rank one,
\[
 \dim T_s=1.
\]
The same remark gives an integer $d\ge1$ and a decomposition
\[
 x^d=yz,
 \qquad y\in T_s(\R),\quad z\in T_a(\R).
\]
The cyclic subgroup generated by $x^d$ is still Zariski-dense in $T$, since
the power map $t\mapsto t^d$ is a surjective isogeny of the torus $T$.
Consequently $y$ cannot be torsion.  Indeed, if $y$ had finite order, then
a further positive power of $x^d$ would lie in $T_a(\R)$, so its cyclic
subgroup would have Zariski closure contained in the proper subtorus $T_a<T$,
contrary to the Zariski density of $\langle x^d\rangle$ in $T$.

Thus $x$ has nontrivial hyperbolic component.  Equivalently, $x$ is not
relatively compact.  Since $x$ is semisimple and $\SU(n,1)$ has real rank
one, a semisimple non-relatively-compact element is loxodromic.  Hence $x$ is
loxodromic.

It remains to verify simplicity of the spectrum in the standard
$(n+1)$-dimensional representation.  Since $x$ is regular semisimple,
$T=Z_{\mathbf H}(x)^\circ$ is a maximal real torus.  After base change to
$\C$,
\[
 \mathbf H_{\C}\simeq \SL_{n+1,\C}
\]
and $T_{\C}$ is a maximal complex torus equal to the identity component of
the centralizer of $x$ in $\mathbf H_{\C}$.  A semisimple element of
$\SL_{n+1}(\C)$ has maximal-torus centralizer if and only if its eigenvalues
in the standard representation are pairwise distinct.  Therefore the
standard spectrum of $x$ is simple.

Thus $x\in G$ is a loxodromic element with pairwise distinct eigenvalues.
\end{proof}

\begin{proposition}[Proposition 2 --- regular loxodromics under complex irreducibility]
\label{prop:regular-lox-irreducible}
Let $G<\SU(n,1)$, with $n\ge2$, be complex irreducible. Then $G$ contains a
loxodromic element whose eigenvalues in the standard representation are pairwise distinct.
\end{proposition}

\begin{proof}
Let $\mathbf Z$ be the real Zariski closure of $G$ in $\mathbf{SU}(n,1)$.
By Ben Ahmed--Zeghib \cite[Corollary~4.3]{BenAhmedZeghib2016}, after conjugation
in $\SU(n,1)$ the identity component of $\mathbf Z(\mathbb R)$ is either
\[
 \SU(n,1)
 \qquad\text{or}\qquad
 \SO^0(n,1).
\]
Let $\mathbf H=\mathbf Z^{\circ}$ be the connected real algebraic identity
component.  Thus $\mathbf H$ is the ambient $\mathbf{SU}(n,1)$ in the first
case and the standard real form $\mathbf{SO}(n,1)$ in the second (with
$\mathbf H(\mathbb R)^\circ=\SO^0(n,1)$).  Put
\[
 G_0=G\cap \mathbf H(\mathbb R).
\]
Because $\mathbf Z/\mathbf H$ is finite, $G_0$ has finite index in $G$.
Moreover a finite-index subgroup has the same identity component of its
Zariski closure; since $G_0\subset\mathbf H(\mathbb R)$ and $\mathbf H$ is
connected, the real Zariski closure of $G_0$ is exactly $\mathbf H$.

The group $\mathbf H$ is connected and semisimple of real rank one.  Apply
Prasad--Rapinchuk \cite[Theorem~1]{PrasadRapinchuk2003} to the Zariski-dense
subgroup $G_0<\mathbf H(\mathbb R)$.  There exists
\[
 x\in G_0
\]
which is regular and $\mathbb R$-regular, and such that the cyclic subgroup
$\langle x\rangle$ is Zariski-dense in the maximal torus
\[
 T=Z_{\mathbf H}(x)^\circ.
\]

We first show that $x$ is loxodromic.  Write $T_s$ for the maximal
$\mathbb R$-split subtorus of $T$.  Since $x$ is $\mathbb R$-regular,
$T_s$ is a maximal $\mathbb R$-split torus of $\mathbf H$; hence
\[
 \dim T_s=1.
\]
Moreover the Zariski density of $\langle x\rangle$ in $T$ implies that its
image in the split quotient of $T$ is Zariski-dense and in particular
nontrivial.  Thus the split component of $x$ is nontrivial.  A semisimple
element of a real-rank-one group $\SU(n,1)$ or $\SO(n,1)$ with nontrivial
split component is loxodromic.

It remains to prove that the eigenvalues of $x$ in the standard complex
representation are pairwise distinct.  After complexification, $T_{\mathbb C}$ is a maximal torus of
$\mathbf H_{\mathbb C}$, and the restriction of the standard representation
to $T_{\mathbb C}$ has one-dimensional weight spaces with distinct weights,
both for
\[
 \mathbf H_{\mathbb C}\simeq \SL_{n+1}(\mathbb C)
\]
in the unitary case and for
\[
 \mathbf H_{\mathbb C}\simeq \SO_{n+1}(\mathbb C)
\]
in the real case.  Suppose two standard eigenvalues of $x$ were equal.  If
$\mu\ne\nu$ are the corresponding torus weights, then
\[
 (\mu\nu^{-1})(x)=1.
\]
Hence the cyclic group $\langle x\rangle$ would be contained in the kernel of
the nontrivial character $\mu\nu^{-1}$, a proper algebraic subgroup of
$T_{\mathbb C}$.  But Zariski density of $\langle x\rangle$ in the real torus
$T$ remains Zariski density after base change to $\mathbb C$: the real and
imaginary parts of any complex regular function vanishing on $\langle x\rangle$
would both vanish on $T$.  This is a contradiction.  Therefore all standard
eigenvalues of $x$ are distinct.

Thus $x\in G$ is a regular loxodromic element, as required.
\end{proof}

Thus the coefficient calculations below are made after choosing such a regular loxodromic element $A$.

\subsection{Inversion and conjugation}

Let $\iota$ be the involution of indices determined by the Siegel form: $\iota(1)=d$, $\iota(d)=1$, and $\iota(j)=j$ for $2\le j\le d-1$.

\begin{lemma}\label{lem:inverse}
If $B=(b_{ij})\in\SU(J)$, then
\[
 (B^{-1})_{ij}=\overline{b_{\iota(j),\iota(i)}}.
\]
In particular,
\[
 \tr(B^{-1})=\overline{\tr(B)}.
\]
\end{lemma}

\begin{proof}
Since $J=J^{-1}=J^*$ and $B^*JB=J$, one has
\[
 B^{-1}=JB^*J.
\]
Left and right multiplication by $J$ apply the index involution $\iota$, so the $(i,j)$ entry is
\[
 (JB^*J)_{ij}=\overline{b_{\iota(j),\iota(i)}}.
\]
Taking the trace gives the final identity.
\end{proof}

\begin{lemma}\label{lem:Kbar}
For a regular loxodromic $A$ with distinct eigenvalues,
\[
 K=\Q(\tr\Gamma,E(A))
\]
is stable under complex conjugation.
\end{lemma}

\begin{proof}
Since $\overline{\tr(\gamma)}=\tr(\gamma^{-1})$ for $\gamma\in\Gamma$, the trace field is conjugation-stable.  The eigenvalues of $A^{-1}$ are the inverses of those of $A$, while unitarity pairs eigenvalues by $\lambda\mapsto\bar\lambda^{-1}$.  Hence conjugates of the eigenvalues of $A$ also lie in $K$.
\end{proof}

\subsection{Vandermonde recovery of diagonal coefficients}

Assume from now on that
\[
 A=\diag(\lambda_1,\ldots,\lambda_d),
\]
with the $\lambda_i$ pairwise distinct.

\begin{lemma}\label{lem:diagonal}
For every $B=(b_{ij})\in\Gamma$ one has
\[
 b_{11},\ldots,b_{dd}\in K.
\]
\end{lemma}

\begin{proof}
The traces
\begin{align*}
 \tr(B)&=b_{11}+\cdots+b_{dd}=t_0,\\
 \tr(AB)&=\lambda_1b_{11}+\cdots+\lambda_db_{dd}=t_1,\\
 &\vdots\\
 \tr(A^{d-1}B)&=\lambda_1^{d-1}b_{11}+\cdots+\lambda_d^{d-1}b_{dd}=t_{d-1}
\end{align*}
lie in $K$.  The coefficient matrix is the Vandermonde matrix
\[
 V=(\lambda_j^m)_{0\le m\le d-1,\ 1\le j\le d}
\]
with
\[
 \det V=\prod_{1\le p<q\le d}(\lambda_q-\lambda_p)\ne0.
\]
Cramer's rule gives $b_{ii}\in K$.
\end{proof}

\subsection{Pair products and unitary mixed products}

The next calculation makes explicit which coefficient combinations are recovered before passing to the invariant formulation.

\begin{lemma}\label{lem:pair-products}
For every $B=(b_{ij})\in\Gamma$ and all $i,j$,
\[
 b_{ij}b_{ji}\in K.
\]
\end{lemma}

\begin{proof}
Consider
\[
 C_k=BA^kB,\qquad 1\le k\le d-1.
\]
For instance,
\[
 (C_k)_{11}=\lambda_1^k b_{11}^2+\sum_{j=2}^d\lambda_j^k b_{1j}b_{j1}.
\]
The diagonal-entry lemma applied to $C_k\in\Gamma$ shows that the left-hand side belongs to $K$.  Therefore the unknowns
\[
 x_j=b_{1j}b_{j1},\qquad 2\le j\le d,
\]
satisfy a reduced Vandermonde system with coefficient matrix
\[
 V_1=(\lambda_j^m)_{1\le m\le d-1,\ 2\le j\le d}.
\]
Its determinant is nonzero because the $\lambda_j$ are nonzero and pairwise distinct.  The same argument with another row index gives the result for all $i,j$.
\end{proof}

\begin{lemma}\label{lem:mixed-products}
For every $B=(b_{ij})\in\Gamma$ and all $1\le i,j\le d$ one has
\[
 b_{ij}\,\overline{b_{\iota(i),\iota(j)}}\in K.
\]
In particular,
\[
 b_{1j}\overline{b_{dj}},\quad
 b_{j1}\overline{b_{jd}}\in K
 \qquad(2\le j\le d-1),
\]
and
\[
 b_{ij}\overline{b_{ij}}\in K
 \qquad(2\le i,j\le d-1).
\]
\end{lemma}

\begin{proof}
For fixed $i$ use
\[
 D_k=BA^kB^{-1},\qquad 0\le k\le d-1.
\]
Lemma~\ref{lem:inverse} gives
\[
 (D_k)_{ii}
 =\sum_{j=1}^d \lambda_j^k b_{ij}
   \overline{b_{\iota(i),\iota(j)}}.
\]
Since $D_k\in\Gamma$, Lemma~\ref{lem:diagonal} gives $(D_k)_{ii}\in K$.  The full Vandermonde matrix $(\lambda_j^k)_{0\le k\le d-1,\,1\le j\le d}$ is invertible, so each coefficient
\[
 b_{ij}\overline{b_{\iota(i),\iota(j)}}
\]
lies in $K$.  The displayed special cases follow because $\iota(j)=j$ for $2\le j\le d-1$.
\end{proof}

\subsection{The word \texorpdfstring{$E(m,k,r)$}{E(m,k,r)}}

To recover individual coefficients explicitly, the same interpolation is applied to a longer word.  Suppose that $\Gamma_0=\langle A,B\rangle$ with $A$ as above.  For the fixed row under consideration, assume that the entries used as divisors below are nonzero and that a diagonal change of basis commuting with $A$ has normalized the indicated inverse entry to $1$.  This is a genuine normalization hypothesis: diagonal conjugation rescales a nonzero entry but cannot make a zero entry nonzero.  The unconditional descent theorem in Section~\ref{sec:linear} will not depend on it.

Write $B^{-1}=(d_{ij})$ and set
\[
 E(m,k,r)=BA^mBA^kBA^rB^{-1}.
\]
Then
\begin{align}
 E(m,k,r)_{ii}
 =\sum_{\ell=1}^d\lambda_\ell^r
 \left[\sum_{w=1}^d\lambda_w^k
 \left(\sum_{s=1}^d\lambda_s^m b_{is}b_{sw}\right)b_{w\ell}\right]d_{\ell i}.
 \label{eq:E-mkr}
\end{align}
Since $E(m,k,r)\in\Gamma_0$, its diagonal entries lie in $K$.  The successive inversions in $r$, $k$, and $m$ are recorded next; separating them from the expansion above avoids repeating the same interpolation argument twice.

\subsection{Three successive Vandermonde inversions}\label{subsec:three-cramer}

We spell out the three inversions because they show directly how individual coefficients arise from traces of longer words.  Write $d=n+1$ and, for fixed $m,k$ and $i$, set
\[
 X_I(m,k)=
 \left[\sum_{w=1}^{d}\lambda_w^k
 \left(\sum_{s=1}^{d}\lambda_s^m b_{is}b_{sw}\right)b_{wI}\right]d_{Ii}.
\]
For $r=1,\ldots,d$ the diagonal coefficient in \eqref{eq:E-mkr} gives
\[
 \begin{pmatrix}
 \lambda_1&\lambda_2&\cdots&\lambda_d\\
 \lambda_1^2&\lambda_2^2&\cdots&\lambda_d^2\\
 \vdots&\vdots&\ddots&\vdots\\
 \lambda_1^d&\lambda_2^d&\cdots&\lambda_d^d
 \end{pmatrix}
 \begin{pmatrix}
 X_1(m,k)\\X_2(m,k)\\\vdots\\X_d(m,k)
 \end{pmatrix}
 =
 \begin{pmatrix}
 E(m,k,1)_{ii}\\E(m,k,2)_{ii}\\\vdots\\E(m,k,d)_{ii}
 \end{pmatrix}.
\]
The determinant of the matrix on the left is
\[
 \left(\prod_{j=1}^{d}\lambda_j\right)
 \prod_{1\le p<q\le d}(\lambda_q-\lambda_p),
\]
which is nonzero.  Thus every $X_I(m,k)$ lies in $K$.

The normalization chooses an index $I_0$ with $d_{I_0i}=1$ and assumes that the entries used in the final division are nonzero.  A diagonal conjugation alone does not impose these conditions simultaneously for every row.  Under the normalized configuration one may take $I_0=d$, and then
\[
 X_d(m,k)=
 \sum_{w=1}^{d}\lambda_w^k
 \left(\sum_{s=1}^{d}\lambda_s^m b_{is}b_{sw}\right)b_{wd}
 \in K.
\]
Thus the divisions in this subsection require the stated non-vanishing assumptions.  Theorem~\ref{thm:linear-descent} is proved below by cyclic projector products and does not depend on this normalization.

Keeping $m$ fixed and varying $k=1,\ldots,d$ gives a second Vandermonde system.  Its unknowns are
\[
 Y_I(m)=\sum_{s=1}^{d}\lambda_s^m b_{is}b_{sI}b_{Id},
\]
and therefore
\[
 Y_I(m)\in K\qquad(1\le I\le d).
\]
In particular,
\[
 Y_1(m)=\sum_{s=1}^{d}\lambda_s^m b_{is}b_{s1}b_{1d}\in K.
\]
Finally, varying $m$ gives the third system
\[
 \begin{pmatrix}
 \lambda_1&\cdots&\lambda_d\\
 \vdots&\ddots&\vdots\\
 \lambda_1^d&\cdots&\lambda_d^d
 \end{pmatrix}
 \begin{pmatrix}z_1\\\vdots\\z_d\end{pmatrix}
 =
 \begin{pmatrix}Y_1(1)\\\vdots\\Y_1(d)\end{pmatrix},
 \qquad
 z_j=b_{ij}b_{j1}b_{1d}.
\]
Thus each $z_j$ belongs to $K$.  Under this normalized configuration, if the factors $b_{j1}b_{1d}$ used in the final division are nonzero and belong to $K^\times$, then the corresponding coefficients $b_{ij}$ are recovered over $K$.  The cyclic-projector argument of Theorem~\ref{thm:linear-descent} removes these normalization requirements from the linear-descent statement.

Although the projector proof below is shorter, the explicit calculation is useful because it shows exactly where the non-vanishing normalization enters and makes clear that the three steps are repeated inversions of the same evaluation map at the spectrum of $A$.

The tensor formulation packages all three inversions at once.  If
\[
 T_{mkr}=(Y_0A^mY_1A^kY_2A^rY_3)_{ab},
\]
then
\[
 T_{mkr}=\sum_{s,u,\ell=1}^{d}
 \lambda_s^m\lambda_u^k\lambda_\ell^r
 (Y_0)_{as}(Y_1)_{su}(Y_2)_{u\ell}(Y_3)_{\ell b}.
\]
The coefficient transform is $V^{\otimes3}$, and
\[
 \det(V^{\otimes3})=(\det V)^{3d^2}\ne0.
\]
More generally,
\[
 \det(V^{\otimes r})=(\det V)^{r d^{r-1}}.
\]
Thus the long Cramer calculation and the tensorial statement are literally the same interpolation mechanism viewed at different resolutions.

\subsection{Spectral projectors and cyclic traces}

For each eigenvalue $\lambda_i$ of $A$ define the spectral projector
\[
 P_i=\prod_{j\ne i}\frac{A-\lambda_jI}{\lambda_i-\lambda_j}\in\M_d(K).
\]
In the eigenbasis of $A$, $P_i$ is the matrix unit $E_{ii}$.

\begin{lemma}\label{lem:cycles}
Let $X_1,\ldots,X_r\in\Gamma$ and let $i_0,\ldots,i_{r-1}$ be indices.  Then
\[
 \tr(P_{i_0}X_1P_{i_1}X_2\cdots P_{i_{r-1}}X_r)
 \in K,
\]
and in the eigenbasis this trace is the cyclic product
\[
 (X_1)_{i_0i_1}(X_2)_{i_1i_2}\cdots(X_r)_{i_{r-1}i_0}.
\]
\end{lemma}

\begin{proof}
Each $P_i$ is a polynomial in $A$ with coefficients in $K$.  Expanding the product reduces the trace to a $K$-linear combination of traces of words in $A$ and the $X_j$.  Since those words belong to the $K$-trace algebra of $\Gamma$, the trace lies in $K$.
\end{proof}

This projector identity is the invariant core of the repeated Vandermonde systems: the explicit systems reconstruct the same cyclic coefficients by interpolation.

\section{Linear descent over the spectral field}\label{sec:linear}

The result of this section is deliberately stated over the full spectral field.  It is a linear descent statement and should not be read as a unitary descent statement.

\begin{theorem}\label{thm:linear-descent}
Let $\Gamma<\SU(J)$ be Zariski-dense and contain a regular loxodromic element
\[
 A=\diag(\lambda_1,\ldots,\lambda_d)
\]
with distinct eigenvalues.  Put
\[
 K=\Q(\tr\Gamma,E(A)).
\]
Then there exists a diagonal $D\in\GL_d(\C)$ commuting with $A$ such that
\[
 D^{-1}\Gamma D\subset\SL_d(K).
\]
\end{theorem}

\begin{proof}
Because $\Gamma$ is Zariski-dense, it is absolutely irreducible.  Choose $B\in\Gamma$ with
\[
 b_{1i}b_{i1}\ne0\qquad(2\le i\le d);
\]
such a choice is possible because these are finitely many non-vanishing regular conditions on a nonempty Zariski-open subset of the ambient group, and $\Gamma$ is Zariski-dense.  Let
\[
 D=\diag(1,b_{21},\ldots,b_{d1}),
 \qquad \widetilde X=D^{-1}XD.
\]
Then $D$ commutes with $A$, and for $\widetilde B=D^{-1}BD$ one has
\[
 \widetilde b_{i1}=1\quad(2\le i\le d),
 \qquad
 \widetilde b_{1i}=b_{1i}b_{i1}\in K^\times
 \quad(2\le i\le d),
\]
where the second assertion is Lemma~\ref{lem:pair-products}.  Cyclic products are invariant under diagonal conjugation, so Lemma~\ref{lem:cycles} may be read directly in the tilded coordinates.  The diagonal entries of every $\widetilde X$ are already in $K$ by Lemma~\ref{lem:diagonal}, since diagonal conjugation does not change them.

For $j\ge2$, Lemma~\ref{lem:cycles}, applied to the two-cycle $(X,B)$ with indices $(1,j)$, gives
\[
 \widetilde x_{1j}\widetilde b_{j1}=\widetilde x_{1j}\in K.
\]
For $i\ge2$, the two-cycle $(B,X)$ with indices $(1,i)$ similarly gives
\[
 \widetilde b_{1i}\widetilde x_{i1}\in K,
\]
and division by $\widetilde b_{1i}\in K^\times$ yields $\widetilde x_{i1}\in K$.  Finally, for $i,j\ge2$, the three-cycle $(B,X,B)$ with indices $(1,i,j)$ gives
\[
 \widetilde b_{1i}\widetilde x_{ij}\widetilde b_{j1}
 =\widetilde b_{1i}\widetilde x_{ij}\in K,
\]
so again $\widetilde x_{ij}\in K$.  Hence every entry of $D^{-1}XD$ lies in $K$ for every $X\in\Gamma$.  Therefore $D^{-1}\Gamma D\subset\M_d(K)$, and determinant one gives the stated inclusion in $\SL_d(K)$.
\end{proof}

\subsection{The two-generator argument written only in equations}\label{subsec:two-generator-equations}

We now isolate the calculation for the case
$\Gamma_0=\langle A,B\rangle$.  This makes explicit which part of the
argument follows directly from the equations above, without invoking a
separate coefficient-recovery argument from
\cite[p.~364]{CunhaGusevskii2014}.

\begin{proposition}[Two-generator descent from the equations]\label{prop:two-generator-equations}
Let
\[
 \Gamma_0=\langle A,B\rangle<\SU(J),\qquad
 A=\diag(\lambda_1,\ldots,\lambda_d),
\]
where the $\lambda_i$ are pairwise distinct, and assume that $\Gamma_0$ is
complex irreducible.  Put
\[
 K=\Q(\tr\Gamma_0,E(A)).
\]
Then there exists a diagonal matrix $D\in\GL_d(\C)$ commuting with $A$ such
that
\[
 D^{-1}\Gamma_0D\subset\SL_d(K).
\]
\end{proposition}

\begin{proof}
Consider the directed graph $\mathcal G(B)$ with vertex set
$\{1,\ldots,d\}$ and an arrow
\[
 j\longrightarrow i
 \quad\Longleftrightarrow\quad
 b_{ij}\ne0.
\]
Because $A$ has simple spectrum, every $A$-invariant subspace is the span of
a subset of its eigenvectors.  If $\mathcal G(B)$ were not strongly
connected, a sink strongly connected component would span a nonzero proper
subspace invariant under both $A$ and $B$, contradicting complex
irreducibility.  Hence $\mathcal G(B)$ is strongly connected.

Fix the vertex $1$.  Choose a directed spanning tree rooted at $1$.  For each
vertex $i$, let
\[
 1=v_0\to v_1\to\cdots\to v_m=i
\]
be the unique tree path and define
\[
 d_1=1,
 \qquad
 d_i=\prod_{r=0}^{m-1} b_{v_{r+1},v_r}\ne0.
\]
Set
\[
 D=\diag(d_1,\ldots,d_d).
\]
For every tree edge $j\to i$ one then has
\[
 (D^{-1}BD)_{ij}=d_i^{-1}b_{ij}d_j=1.
\]

Strong connectivity also gives, for each $i$, a directed path from $i$ back
to $1$.  Let $c_i$ be the nonzero product of the corresponding entries of
$B$ along such a path; take $c_1=1$.  The tree path from $1$ to $i$ followed
by this return path is a directed cycle.  Lemma~\ref{lem:cycles} therefore
gives
\[
 d_i c_i\in K^\times.
\]

Now let $X=(x_{ij})\in\Gamma_0$.  Follow the tree path from $1$ to $j$, then
use the single $X$-edge $j\to i$, and finally follow the chosen $B$-path from
$i$ back to $1$.  The corresponding mixed cyclic product is
\[
 d_j\,x_{ij}\,c_i.
\]
By Lemma~\ref{lem:cycles}, this product belongs to $K$.  Dividing by the
nonzero cycle product $d_i c_i\in K^\times$ yields
\[
 (D^{-1}XD)_{ij}
 =d_i^{-1}x_{ij}d_j
 =\frac{d_jx_{ij}c_i}{d_ic_i}
 \in K.
\]
This holds for every $i,j$ and every $X\in\Gamma_0$.  Thus
\[
 D^{-1}\Gamma_0D\subset\M_d(K).
\]
Since all elements have determinant one, the image is contained in
$\SL_d(K)$.
\end{proof}

\begin{remark}\label{rem:two-generator-gap}
The proof above removes the case-by-case non-vanishing analysis used in the
three-dimensional argument of Cunha--Gusevskii \cite{CunhaGusevskii2014}.  Complex irreducibility is
used only once: it makes the directed nonzero-entry graph of $B$ strongly
connected.  The rest of the proof consists entirely of the cyclic trace
identities already established above.
\end{remark}

\begin{remark}[What the equations prove, and what remains for $K_L$]\label{rem:KL-descent}
The preceding proposition is an unconditional coefficient calculation over
\[
 K_{\mathrm{spec}}=\Q(\tr\Gamma_0,E(A)).
\]
It proves linear descent
\[
 D^{-1}\Gamma_0D\subset\SL_d(K_{\mathrm{spec}})
\]
under the stated hypotheses.  Replacing $K_{\mathrm{spec}}$ by
\[
 K_L=\Q(\tr\Gamma_0,L)
\]
requires an additional spectral step.  Lemma~\ref{lem:extreme-eigenvalues}
puts the two off-circle eigenvalues in $K_L$, but in arbitrary dimension the
remaining unit-circle eigenvalues have not yet been shown to lie in $K_L$.
Consequently the full projectors
\[
 P_i=\prod_{j\ne i}\frac{A-\lambda_jI}{\lambda_i-\lambda_j}
\]
are not yet known to be defined over $K_L$.

There is a second, logically independent step.  Even after a linear model over
$K_L$ has been obtained, concluding
\[
 \Gamma_0\subset\SU(J,K_L)
\]
requires the preserved Hermitian form to be $K_L$-similar to the standard form
$J$.  Sections~\ref{sec:hermitian} and the obstruction in
Theorem~\ref{thm:one-A-obstruction} show why this unitary descent cannot be
silently identified with linear descent.  Thus the proposed theorem over
$\Q(\Gamma,L)$ naturally splits into a spectral/linear problem and a separate
Hermitian-similarity problem.
\end{remark}

\subsection{Real trace fields and real forms}\label{subsec:real-trace}

The real-trace case behaves differently from the general unitary descent problem: complex irreducibility already forces a real form, with no discreteness assumption.

\begin{theorem}[Real trace field without discreteness]\label{thm:real-trace-no-discrete}
Let $G<\SU(n,1)$, with $n\ge2$, be non-elementary and complex irreducible.  If
\[
 \Q(\tr G)\subset\R,
\]
then $G$ is conjugate in $\SU(n,1)$ to a subgroup of $\SO(n,1)$.  Conversely, every subgroup of $\SO(n,1)$ has real trace field.
\end{theorem}

\begin{proof}
Let $V=\C^{n+1}$ with its standard Hermitian form $h$ of signature $(n,1)$, and let $\rho:G\to\SU(h)$ be the defining representation.  Since every trace is real,
\[
 \chi_{\rho}(g)=\overline{\chi_{\rho}(g)}=\chi_{\bar\rho}(g)
 \qquad(g\in G).
\]
The representations $\rho$ and $\bar\rho$ are irreducible and have the same character.  By the Brauer--Nesbitt theorem they are equivalent.  Hence there is an invertible antilinear map
\[
 \mathcal J:V\longrightarrow V
\]
commuting with $G$.

By Schur's lemma, $\mathcal J^2=cI$ for some $c\in\C^\times$.  Since $\mathcal J$ is antilinear, comparison of $\mathcal J^2\mathcal J$ and $\mathcal J\mathcal J^2$ gives $c=\bar c$, so $c\in\R^\times$.  After rescaling $\mathcal J$ we may assume
\[
 \mathcal J^2=I
 \qquad\text{or}\qquad
 \mathcal J^2=-I.
\]
The form
\[
 h_{\mathcal J}(v,w)=\overline{h(\mathcal Jv,\mathcal Jw)}
\]
is again $G$-invariant.  Complex irreducibility implies that $h_{\mathcal J}=a h$ for some $a\in\R^\times$.  Since the two forms have the same signature $(n,1)$ and $n\ge2$, one has $a>0$.  Rescaling $\mathcal J$ by $a^{-1/2}$ therefore gives
\[
 h(\mathcal Jv,\mathcal Jw)=\overline{h(v,w)}.
\]

The alternative $\mathcal J^2=-I$ is impossible.  If $v$ is negative, then $\mathcal Jv$ is also negative and the antiunitary identity, applied to $v$ and $\mathcal Jv$, gives
\[
 h(v,\mathcal Jv)=0.
\]
Thus $v$ and $\mathcal Jv$ span a two-dimensional negative complex subspace, contradicting the fact that the negative index of $h$ is one.  Consequently $\mathcal J^2=I$.

The fixed-point space
\[
 V_{\R}=\{v\in V:\mathcal Jv=v\}
\]
is a real form of $V$, is preserved by $G$, and carries the real symmetric bilinear form obtained by restricting $h$.  This form has signature $(n,1)$.  Choosing a real orthonormal basis identifies the action of $G$ with a subgroup of $O(n,1)$; determinant one places the image in $\SO(n,1)$.  The conjugating change of basis may be chosen in $\SU(n,1)$ after multiplying it by a scalar of modulus one, which does not change the conjugation.

The converse is immediate because matrices in $\SO(n,1)$ have real entries and therefore real traces.
\end{proof}

\begin{remark}
Kim and Kim proved the analogous conclusion for irreducible \emph{discrete} subgroups of $\SU(n,1)$ as part of their real-trace-field rigidity theorem \cite[Theorem~1.6]{KimKim2016}.  The argument above follows a representation-theoretic route and shows that, for the complex statement, discreteness is unnecessary.
\end{remark}

There is a complementary trace-algebra proof.  Let
\[
 \mathcal A_K=\Span_K(\Gamma)\subset\M_d(\C).
\]
The argument is useful because it removes the dependence on a chosen first column.

\begin{proposition}\label{prop:burnside-procesi-en}
Assume $\Gamma<\SL_d(\C)$ is absolutely irreducible and $K$ contains $\tr(\gamma)$ for every $\gamma\in\Gamma$.  Then there exist elements
\[
 S_1,\ldots,S_{d^2}\in\Gamma
\]
forming a complex basis of $\M_d(\C)$ such that every $\gamma\in\Gamma$ can be written uniquely as
\[
 \gamma=\sum_{\alpha=1}^{d^2}c_\alpha(\gamma)S_\alpha,
\]
with all coefficients $c_\alpha(\gamma)$ in the field generated by the traces of words in $\Gamma$.
\end{proposition}

\begin{proof}
Burnside's theorem gives $\Span_\C(\Gamma)=\M_d(\C)$, so we may select $d^2$ group elements forming a complex basis.  Let
\[
 G_{\alpha\beta}=\tr(S_\alpha S_\beta).
\]
The trace pairing on $\M_d(\C)$ is non-degenerate, hence the Gram matrix $G$ is invertible.  If
\[
 \gamma=\sum_\alpha c_\alpha S_\alpha,
\]
then for every $\beta$
\[
 \tr(\gamma S_\beta)=\sum_\alpha c_\alpha G_{\alpha\beta}.
\]
Both the entries of $G$ and the right-hand trace vector are traces of group elements, hence belong to the trace field.  Cramer's rule gives the claim.
\end{proof}

It follows that $\mathcal A_K$ is a central simple $K$-algebra of degree $d$.  The regular loxodromic element provides more information.

\begin{proposition}\label{prop:projectors-split-en}
Under the hypotheses of Theorem~\ref{thm:linear-descent}, the algebra $\mathcal A_K$ is split:
\[
 \mathcal A_K\simeq\M_d(K).
\]
\end{proposition}

\begin{proof}
The spectral projectors
\[
 P_i=\prod_{j\ne i}\frac{A-\lambda_jI}{\lambda_i-\lambda_j}
\]
belong to $K[A]\subset\mathcal A_K$.  They are pairwise orthogonal primitive idempotents with sum $1$, and after extension to $\C$ each has matrix rank one.  A central simple algebra of degree $d$ containing a complete set of $d$ mutually orthogonal primitive idempotents has a minimal right ideal of $K$-dimension $d$ and therefore has Schur index one.  Hence it is isomorphic to $\M_d(K)$.
\end{proof}

This gives the same linear descent without choosing a coefficient normalization.  The Cramer argument retains the explicit formulas, while the trace-algebra argument identifies their invariant algebraic content.

\begin{remark}
Theorem~\ref{thm:linear-descent} gives a $K$-rational \emph{linear} model.  Passage to the standard unitary form is equivalent to a separate Hermitian similarity problem.
\end{remark}

\section{Hermitian descent and the norm obstruction}\label{sec:hermitian}

Linear descent gives a $K$-model of the representation, but it leaves open which Hermitian form that model preserves.  The remaining problem is therefore one of Hermitian similarity rather than matrix coefficients.

Let
\[
 \rho(\Gamma)=D^{-1}\Gamma D\subset\SL_d(K)
\]
be the representation of Theorem~\ref{thm:linear-descent}.  Since the original representation is unitary, $\rho(\Gamma)$ preserves the Hermitian form
\[
 H=D^*JD.
\]
Assume henceforth that complex conjugation acts nontrivially on $K$ and put
\[
 K_0=K^{\bar{\ }}.
\]
Thus $K/K_0$ is quadratic.  This is the regime used in all arithmetic applications below.  If $K=K_0$, one may pass to the conjugation-stable quadratic extension $K(i)/K$ and apply the argument there, but no descent of the Hermitian form to the original real field $K$ is asserted.  Under the quadratic hypothesis, absolute irreducibility implies that the space of invariant Hermitian forms is one-dimensional over $K_0$, and trace duality descends the form itself.

\begin{theorem}\label{thm:hermitian-descent}
Assume $d=n+1\ge3$ and that complex conjugation acts nontrivially on $K$, with fixed field $K_0$.  Then there exists a non-degenerate Hermitian matrix
\[
 H_K\in\M_d(K),\qquad H_K^*=H_K,
\]
such that
\[
 \rho(\gamma)^*H_K\rho(\gamma)=H_K
 \qquad(\gamma\in\Gamma).
\]
After multiplying by an element of $K_0^\times$ and making a diagonal $K$-change of basis, one may write
\[
 H_K\simeq
 \mathbb H\oplus\langle1,q_3,\ldots,q_n\rangle,
 \qquad q_j\in K_0^\times,
\]
where
\[
 \mathbb H=\begin{pmatrix}0&1\\1&0\end{pmatrix}.
\]
Moreover, $\Gamma$ admits a realization in $\SU(J,K)$ if and only if $H_K$ is similar over $K$ to $J$.
\end{theorem}

\begin{proof}
For each $\gamma\in\Gamma$ let
\[
 W_\gamma=
 \{X\in\M_d(\C):\rho(\gamma)^*X\rho(\gamma)=X\}.
\]
Because $\rho(\gamma)\in\M_d(K)$ and $K$ is stable under complex conjugation, each $W_\gamma$ is the complexification of a $K$-linear subspace of $\M_d(K)$.  The space of forms invariant under the whole group is
\[
 W=\bigcap_{\gamma\in\Gamma}W_\gamma.
\]
Absolute irreducibility and Schur's lemma give $\dim_\C W=1$: after fixing one nonzero invariant Hermitian form $H$, every invariant sesquilinear form is a scalar multiple of $H$.

Since $\M_d(\C)$ is finite-dimensional, an arbitrary intersection of linear subspaces is already the intersection of a finite subfamily.  Thus there exist
\[
 \gamma_1,\ldots,\gamma_m\in\Gamma
\]
such that
\[
 W=\bigcap_{j=1}^mW_{\gamma_j}.
\]
This finite homogeneous linear system is defined over $K$ and has complex solution dimension one.  Its $K$-solution space therefore also has dimension one and contains a nonzero generator
\[
 X_0\in\M_d(K).
\]

The line is stable under the involution $X\mapsto X^*$.  Thus
\[
 X_0^*=cX_0
\]
for some $c\in K^\times$, and applying $*$ twice gives $c\bar c=1$.  Hilbert's Theorem~90 provides $a\in K^\times$ with $c=a/\bar a$.  Therefore
\[
 H_K=aX_0
\]
satisfies $H_K^*=H_K$.  It is non-degenerate because, after extension of scalars to $\C$, it spans the same invariant line as the original non-degenerate Hermitian form.  A regular loxodromic element pairs its attracting and repelling eigenlines and leaves the remaining directions orthogonal; the explicit calculation in Subsection~\ref{subsec:normal-form-en} gives the displayed normal form over $K$.

If $H_K\sim_KJ$, a $K$-change of basis takes the $K$-rational representation into $\SU(J,K)$.  The converse is the content of Proposition~\ref{prop:exact-unitary-en}, proved below with the $K$-rational intertwiner made explicit.
\end{proof}

\subsection{Normal form associated with a regular loxodromic element}\label{subsec:normal-form-en}

The form $H_K$ can be described more explicitly before invoking Landherr.  Order the eigenvalues of $A$ so that
\[
 \lambda_d=\bar\lambda_1^{-1},
 \qquad |\lambda_j|=1\quad(2\le j\le d-1).
\]
The equation $A^*H_KA=H_K$ gives, entry by entry,
\[
 (\bar\lambda_i\lambda_j-1)h_{ij}=0.
\]
Because the spectrum is simple, the only nonzero entry in the first row away from the diagonal occurs in the $d$-th column, and the intermediate eigenspaces are mutually orthogonal.  Hence
\[
 H_K=
 \begin{pmatrix}
 0&0&\cdots&0&a\\
 0&q_2&&&0\\
 \vdots&&\ddots&&\vdots\\
 0&&&q_{d-1}&0\\
 \bar a&0&\cdots&0&0
 \end{pmatrix},
 \qquad a\in K^\times,
 \quad q_j\in K_0^\times.
\]
After multiplying the whole form by $q_2^{-1}$ and rescaling the last basis vector by $q_2/a$, one obtains
\[
 H_K\simeq \mathbb H\oplus\langle1,q_3,\ldots,q_{d-1}\rangle.
\]
This normal form is the precise bridge between the matrix extraction and the arithmetic classification of Hermitian forms.

\begin{proposition}\label{prop:exact-unitary-en}
Let $\rho(\Gamma)\subset\SL_d(K)$ be absolutely irreducible and preserve $H_K$.  The following are equivalent:
\begin{enumerate}[label=(\alph*)]
\item there is a $K$-change of basis $P\in\GL_d(K)$ such that
\[
 P^{-1}\rho(\Gamma)P\subset\SU(J,K);
\]
\item there exist $P\in\GL_d(K)$ and $c\in K_0^\times$ such that
\[
 P^*JP=cH_K;
\]
\item $H_K$ and $J$ represent the same Hermitian similarity class over $K/K_0$.
\end{enumerate}
\end{proposition}

\begin{proof}
If (b) holds, then for every $\gamma\in\Gamma$,
\[
 (P^{-1}\rho(\gamma)P)^*J(P^{-1}\rho(\gamma)P)=J,
\]
so (a) follows.  Conversely, if (a) holds, then $P^*JP$ is an invariant Hermitian form for $\rho(\Gamma)$.  Absolute irreducibility makes the invariant Hermitian line one-dimensional over $K_0$, hence
\[
 P^*JP=cH_K
\]
for some $c\in K_0^\times$.  Thus (a) implies (b).  Finally, (b) and (c) are exactly the definition of similarity of Hermitian forms.
\end{proof}

A useful consequence is that the field problem has been reduced to a classical arithmetic question.  No further manipulation of individual matrix entries can remove an obstruction once it has become a nontrivial similarity class of $H_K$.

\subsection{Landherr invariants and parity}

Assume now that $K/K_0$ is a quadratic extension of number fields.  Landherr's theorem classifies Hermitian forms by their local invariants \cite{Landherr1935,Jacobowitz1962,Scharlau1985}.  In odd dimension there is no finite discriminant obstruction once signatures agree.  In even dimension there is one class in
\[
 K_0^\times/N_{K/K_0}(K^\times).
\]

\begin{theorem}\label{thm:landherr}
Let $h$ and $h'$ be non-degenerate Hermitian forms of the same rank over $K/K_0$.  Suppose their signatures agree at every real place of $K_0$ that remains complex in $K$.  If the rank is odd, $h$ and $h'$ are similar.  If the rank is even, their similarity class differs by a single discriminant ratio
\[
 \delta(h,h')\in K_0^\times/N_{K/K_0}(K^\times),
\]
and $h\sim h'$ if and only if this class is trivial.
\end{theorem}

For imaginary quadratic $K/\Q$, the archimedean signature is fixed by the distinguished embedding.  Thus in matrix dimension four the entire obstruction is a rational norm class.

\begin{theorem}\label{thm:normcriterion}
Let $K/\Q$ be imaginary quadratic and
\[
 H(q)=\mathbb H\oplus\langle1,q\rangle,
 \qquad q\in\Q_{>0}.
\]
Then
\[
 H(q)\sim_K J
 \quad\Longleftrightarrow\quad
 q\in N_{K/\Q}(K^\times).
\]
\end{theorem}

\begin{proof}
The two forms have the same signature at the unique archimedean place.  Moreover,
\[
 \frac{\det H(q)}{\det J}=q
 \quad\text{in }\Q^\times/N_{K/\Q}(K^\times).
\]
The even-rank case of Theorem~\ref{thm:landherr} therefore says that the two forms are similar exactly when the class of $q$ is trivial in the norm quotient.
\end{proof}

\begin{example}\label{ex:H3}
Let
\[
 K=\Q(i),\qquad K_0=\Q,
 \qquad
 H_3=\mathbb H\oplus\langle1,3\rangle,
 \qquad
 J=\mathbb H\oplus\langle1,1\rangle.
\]
Put
\[
 S=\diag(1,1,\sqrt3,1).
\]
Then $S^*JS=H_3$.  The element
\[
 A=\diag\left(2,i,-i,\frac12\right)
\]
lies in $\SU(H_3,K)$ and is regular loxodromic.  Since $3$ is not a norm from $\Q(i)$, Theorem~\ref{thm:normcriterion} gives
\[
 H_3\not\sim_KJ,
 \qquad
 H_3\sim_{K(\sqrt3)}J.
\]
\end{example}

The form obstruction can be realized by a Zariski-dense group while keeping the one-element field equal to $K$.

\begin{theorem}\label{thm:one-A-obstruction}
Let
\[
 \Lambda=\SU(H_3,\mathbb Z[i]),
 \qquad
 \Gamma_0=\langle\Lambda,A\rangle<\SU(H_3,K),
 \qquad
 \Gamma=S\Gamma_0S^{-1}<\SU(J).
\]
Then $\Gamma_0$ and $\Gamma$ are finitely generated and Zariski-dense, $A\in\Gamma$, and
\[
 \Q(\tr\Gamma,E(A))=\Q(i).
\]
Nevertheless there is no $U\in\SU(J)$ such that
\[
 U\Gamma U^{-1}\subset\SU(J,\Q(i)).
\]
On the other hand,
\[
 \Gamma\subset\SU\bigl(J,\Q(i,\sqrt3)\bigr).
\]
Thus, for this group and this regular loxodromic element, the one-element field does not remove the Hermitian obstruction, whereas a quadratic spectral extension does.
\end{theorem}

\begin{proof}
The arithmetic group $\Lambda$ is finitely generated and Zariski-dense in the special unitary algebraic group associated with $H_3$ by Borel--Harish-Chandra and Borel density \cite{BorelHC1962,Borel1960}.  Hence $\Gamma_0$ and its conjugate $\Gamma$ are finitely generated and Zariski-dense.  Since $S$ commutes with $A$, we have $A\in\Gamma$.  Traces are unchanged by conjugation and every element of $\Gamma_0$ has coefficients in $K$, so
\[
 \Q(\tr\Gamma)\subset K.
\]
Because $i\in E(A)$, this gives
\[
 \Q(\tr\Gamma,E(A))=K.
\]
The inclusion over $K(\sqrt3)$ is immediate from $S\in\GL_4(K(\sqrt3))$.

Suppose now that there were $U\in\SU(J)$ with $U\Gamma U^{-1}\subset\SU(J,K)$.  The two $K$-rational representations of $\Gamma_0$
\[
 \rho_1(g)=g,
 \qquad
 \rho_2(g)=USgS^{-1}U^{-1}
\]
are equivalent over $\C$.  Since $\Gamma_0$ is Zariski-dense, $\rho_1$ is absolutely irreducible.  The intertwining equations are linear over $K$ and their complex solution space is a line, so there is an invertible $P\in\GL_4(K)$ and a scalar $c\in\C^\times$ with
\[
 US=cP.
\]
Using $U^*JU=J$ and $S^*JS=H_3$ gives
\[
 P^*JP=rH_3
\]
for some $r>0$.  Since both matrices have entries in $K$, a nonzero entry shows $r\in K\cap\R=\Q$.
Taking determinants and using $\det J=-1$ and $\det H_3=-3$ yields
\[
 N_{K/\Q}(\det P)=3r^4.
\]
Every rational norm from $\Q(i)$ has even valuation at every prime $p\equiv3\pmod4$, whereas
\[
 v_3(3r^4)=1+4v_3(r)
\]
is odd.  This contradiction proves the assertion.
\end{proof}

This example shows why linear descent to $\SL_d(K)$ and unitary descent to $\SU(J,K)$ must be treated separately.

\subsection{The Gaussian model}\label{subsec:gaussian-controls-en}

The form $H_3$ admits a convenient integral family of unitary matrices.  The following matrices lie in $\SU(H_3,\mathbb Z[i])$:
\[
 N_1=\begin{pmatrix}
 1&-1+i&0&-1\\0&1&0&1+i\\0&0&1&0\\0&0&0&1
 \end{pmatrix},
 \quad
 N_2=\begin{pmatrix}
 1&0&-3+3i&-3\\0&1&0&0\\0&0&1&1+i\\0&0&0&1
 \end{pmatrix},
\]
\[
 V=\begin{pmatrix}
 1&0&0&i\\0&1&0&0\\0&0&1&0\\0&0&0&1
 \end{pmatrix},
 \quad
 M=\diag(1,i,-i,1),
 \quad
 R=\begin{pmatrix}
 0&0&0&1\\0&-1&0&0\\0&0&1&0\\1&0&0&0
 \end{pmatrix}.
\]
Each identity $X^*H_3X=H_3$ and $\det X=1$ is immediate by direct multiplication.  The matrices $N_1,N_2,V$ are Heisenberg-type translations in the Siegel model, $M$ is a compact rotation, and $R$ exchanges the two isotropic directions.  Their role here is not to assert a new presentation theorem, but to provide an exact arithmetic laboratory in which the distinction between the coefficient field and the standard unitary form can be tested.

The same model also shows why the total spectral field can be strictly more useful than the field attached to one loxodromic element.  Consider
\[
 C=
 \begin{pmatrix}
 49&0&168&-96\\
 0&1&0&0\\
 28&0&97&-56\\
 -24&0&-84&49
 \end{pmatrix}.
\]
A direct calculation gives
\[
 C^*H_3C=H_3,
 \qquad \det C=1,
\]
and
\[
 \chi_C(X)=(X-1)^2(X^2-194X+1).
\]
The nontrivial eigenvalues are
\[
 97+56\sqrt3,
 \qquad
 97-56\sqrt3.
\]

\begin{proposition}\label{prop:total-H3-example}
For the group $\Gamma$ of Theorem~\ref{thm:one-A-obstruction},
\[
 \Q(i,\sqrt3)\subset F_{\rm eig}(\Gamma).
\]
In particular the total spectral field is a unitary field of definition for this example.
\end{proposition}

\begin{proof}
We have $C\in\Lambda\subset\Gamma_0$, hence $SCS^{-1}\in\Gamma$ and has the same eigenvalues as $C$.  Thus $\sqrt3\in F_{\rm eig}(\Gamma)$, while $i$ already occurs in the spectrum of $A$.  Since $S$ has coefficients in $\Q(i,\sqrt3)$ and $\Gamma=S\Gamma_0S^{-1}$ with $\Gamma_0$ defined over $\Q(i)$, the standard unitary realization follows.
\end{proof}

The matrix $C$ reflects a general norm-torus mechanism.  Let
\[
 u=e_1-\frac12e_d.
\]
On the plane spanned by $u$ and a positive vector $e_j$ with Hermitian coefficient $q_j$, the form is $\diag(-1,q_j)$.  If
\[
 a,b\in K_0,
 \qquad a^2-q_jb^2=1,
\]
then
\[
 R_j(a,b)=
 \begin{pmatrix}
  a&q_jb\\
  b&a
 \end{pmatrix}
\]
preserves $\diag(-1,q_j)$ and has eigenvalues
\[
 a+b\sqrt{q_j},
 \qquad
 a-b\sqrt{q_j}.
\]
Thus a nontrivial element of this norm-one torus makes the square root that trivializes $q_j$ visible in the spectrum.

At this point the descent problem has been reduced to its arithmetic core.  The linear model over $K$ is available; what can still fail is the similarity class of its invariant Hermitian form, measured in dimension four by finitely many local norm classes.  Additional spectra matter precisely because they can enlarge the coefficient field enough to trivialize those classes.  The next sections study when the group forces such spectral extensions.

\section{From one spectrum to the total spectral field}\label{sec:total}

A single regular spectrum controls the linear realization, but it need not determine the unitary similarity class.  Passing to all spectra is the natural next step.  Define
\[
 F_{\rm eig}(\Gamma)=K\bigl(E(\gamma):\gamma\in\Gamma\bigr).
\]
The first point is that one maximal torus does not remember all Hermitian embedding data.

Let $E$ be an etale $K_0$-algebra of degree $d$ and put
\[
 L_E=K\otimes_{K_0}E.
\]
For $\lambda\in E^\times$ define
\[
 h_\lambda(x,y)=\operatorname{Tr}_{L_E/K}(\lambda x y^c),
\]
where $c$ extends the nontrivial automorphism of $K/K_0$.

\begin{proposition}\label{prop:blindness}
Fix $E$ and let $\lambda,\lambda'\in E^\times$.  For every
\[
 x\in L_E^1=\{x\in L_E^\times:xx^c=1\},
\]
the multiplication map $m_x:y\mapsto xy$ preserves both $h_\lambda$ and $h_{\lambda'}$.  Its characteristic polynomial and its eigenvalues are independent of $\lambda$.  On the other hand, the Hermitian similarity class may depend on $\lambda$; in particular
\[
 \det V_{E,\lambda}
 =N_{E/K_0}(\lambda)\det_{K_0}(E)
 \quad\text{in }K_0^\times/N_{K/K_0}(K^\times).
\]
\end{proposition}

\begin{proof}
For $x\in L_E^1$ one has $xx^c=1$, and commutativity of $L_E$ gives
\[
 h_\lambda(xu,xv)
 =\Tr_{L_E/K}(\lambda x u v^c x^c)
 =h_\lambda(u,v).
\]
Thus the same multiplication operator $m_x$ is unitary for every parameter $\lambda$, and its characteristic polynomial depends only on multiplication by $x$ on the fixed $K$-algebra $L_E$.

For the determinant, choose a $K$-basis of $L_E$.  Relative to the trace form with parameter $1$, inserting $\lambda$ amounts to composing one argument with multiplication by $\lambda$.  The determinant of that multiplication operator is $N_{E/K_0}(\lambda)$ after passage to the fixed field.  Hence, modulo $N_{K/K_0}(K^\times)$,
\[
 \det V_{E,\lambda}
 =N_{E/K_0}(\lambda)\det V_{E,1},
\]
and we denote the latter fixed class by $\det_{K_0}(E)$.
\end{proof}

\begin{remark}
Thus the splitting field of a generic element sees the abstract torus and the Galois action on its character group, but it does not formally determine the Hermitian embedding parameter.  This is the obstruction to proving Conjecture~\ref{conj:total-main} from generic-torus theory alone; compare \cite{PrasadRapinchuk2010,PrasadRapinchuk2017,PisolkarRajan2016,BoradeEtAl2025}.
\end{remark}

\begin{proposition}\label{prop:finite-landherr}
Let $K/K_0$ be quadratic number fields.  To decide whether $F_{\rm eig}(\Gamma)$ is a unitary field of definition, it suffices to kill:
\begin{enumerate}[label=(\alph*)]
\item the finitely many incompatible archimedean signatures, if any;
\item in even matrix dimension, the finitely supported local discriminant class of $H_K$ relative to $J$.
\end{enumerate}
Equivalently, it is enough to find a finite conjugation-stable extension
\[
 M\subset F_{\rm eig}(\Gamma)
\]
over which these local invariants vanish.
\end{proposition}

\begin{proof}
Landherr's theorem is local--global: after a conjugation-stable finite base change, similarity of $H_K$ and $J$ is determined by the archimedean signatures and, in even rank, by the local discriminant classes.  Only finitely many finite places support the latter obstruction.  Thus if a finite conjugation-stable extension $M\subset F_{\rm eig}(\Gamma)$ kills all of these invariants, then $H_K\sim_MJ$.  Similarity persists after further scalar extension, so $H_K\sim_{F_{\rm eig}(\Gamma)}J$.
\end{proof}

A stronger spectral formulation is the following.

\begin{conjecture}\label{conj:strongnorm}
Write
\[
 H_K\simeq\mathbb H\oplus\langle1,q_3,\ldots,q_n\rangle.
\]
For each $j$ there exists a regular semisimple $\gamma_j\in\Gamma$ such that, with
\[
 F_j=K(E(\gamma_j)),\qquad F_{j,0}=F_j^{\bar{\ }},
\]
one has
\[
 q_j\in N_{F_j/F_{j,0}}(F_j^\times).
\]
\end{conjecture}

\begin{proposition}
Conjecture~\ref{conj:strongnorm} implies that $F_{\rm eig}(\Gamma)$ is a unitary field of definition.
\end{proposition}

\begin{proof}
For each $j$, Conjecture~\ref{conj:strongnorm} gives $x_j\in F_j^\times\subset F_{\rm eig}(\Gamma)^\times$ with
\[
 q_j=x_j\overline{x_j}.
\]
Thus every diagonal coefficient $q_j$ becomes a norm for the quadratic involution on the total spectral field.  The hyperbolic plane is already split, so the exact unitary descent criterion, equivalently Landherr after base change, identifies $H_K$ with the standard form over $F_{\rm eig}(\Gamma)$.
\end{proof}

\subsection{Arithmetic groups}

For arithmetic groups, norm-one tori provide enough integral points to prove a stronger conclusion.

\begin{theorem}\label{thm:arith}
Let $K/K_0$ be a quadratic extension of number fields and let $\Gamma$ be an arithmetic subgroup of the special unitary group associated with
\[
 H_K\simeq\mathbb H\oplus\langle1,q_3,\ldots,q_n\rangle.
\]
Assume that the distinguished real embedding has signature $(n,1)$.  Then
\[
 \sqrt{q_3},\ldots,\sqrt{q_n}\in F_{\rm eig}(\Gamma).
\]
Consequently $F_{\rm eig}(\Gamma)$ is a unitary field of definition.
\end{theorem}

\begin{proof}
Fix $j$.  If $q_j$ is already a square in $K_0$, there is nothing to prove.  Otherwise set
\[
 L_j=K_0(\sqrt{q_j}).
\]
The matrices $R_j(a,b)$ above define a $K_0$-torus
\[
 \mathbf T_j\simeq R^1_{L_j/K_0}\mathbf G_m
\]
acting on $\langle u,e_j\rangle$ and trivially on its orthogonal complement.  Since $q_j$ is positive at the distinguished real place, $\mathbf T_j$ is split there.  The intersection of an arithmetic subgroup with a $K_0$-defined algebraic subgroup is arithmetic in that subgroup, and the Dirichlet theorem for algebraic tori gives positive rank for $\Gamma\cap\mathbf T_j(K_0)$; see \cite{PlatonovRapinchuk1994}.  In particular this intersection contains an element represented by $R_j(a,b)$ with $b\ne0$.

Its spectrum contains
\[
 a+b\sqrt{q_j},
 \qquad
 a-b\sqrt{q_j}.
\]
Because $a,b\in K_0\subset F_{\rm eig}(\Gamma)$ and $b\ne0$,
\[
 \sqrt{q_j}
 =\frac{(a+b\sqrt{q_j})-a}{b}
 \in F_{\rm eig}(\Gamma).
\]
Repeating this for all $j$ puts every diagonal coefficient $q_j$ in the relevant norm group after base change.  The exact unitary descent criterion then identifies $H_K$ with the standard form over $F_{\rm eig}(\Gamma)$.
\end{proof}

\begin{remark}
The torus element used in the proof need not be regular semisimple in the \emph{ambient} unitary group: it acts trivially on the orthogonal complement of the two-dimensional norm plane.  Ambient regularity is unnecessary for the arithmetic descent theorem; the only required datum is the appearance of $\sqrt{q_j}$ in the spectrum.
\end{remark}

\section{Local criteria for removing the Hermitian obstruction}\label{sec:open}

Once the obstruction is written as a norm class, it can be tested place by place.  Spectral extensions of even local degree are enough to kill the nontrivial local invariant, and openness provides a way to realize such extensions by elements of the group.

Assume now that $K/\Q$ is imaginary quadratic and the matrix dimension is even.  Write
\[
 H_K\simeq\mathbb H\oplus\langle1,q_3,\ldots,q_n\rangle,
 \qquad
 \Delta=\prod_{j=3}^{n}q_j\in\Q_{>0}.
\]
Let
\[
 S_\Delta=
 \{p<\infty:\Delta\notin N_{K_p/\Q_p}(K_p^\times)\}.
\]
This is finite.

\begin{lemma}\label{lem:closure-dichotomy}
Let $\mathbf G=\SU(H_K)$ over $\Q$ and let $\Lambda<\mathbf G(\Q)$ be Zariski-dense.  For every prime $p$, if the closure
\[
 \overline\Lambda^{\,p}<\mathbf G(\Q_p)
\]
is not discrete, then it is open.
\end{lemma}

\begin{proof}
Let $H$ be the closure and $\mathfrak h$ its $p$-adic Lie algebra.  Since $\Lambda\subset H$, the subspace $\mathfrak h$ is $\Ad(\Lambda)$-stable.  Zariski density upgrades this to $\Ad(\mathbf G)$-stability.  The Lie algebra of the absolutely almost simple group $\mathbf G$ is simple, hence $\mathfrak h=0$ or $\mathfrak h=\mathfrak g$.  Non-discreteness excludes $\mathfrak h=0$, and full Lie algebra implies openness.
\end{proof}

\begin{lemma}\label{lem:biquad}
For every $p\in S_\Delta$ and every identity neighborhood $\Omega\subset\mathbf G(\Q_p)$ there exists a regular semisimple $t_p\in\Omega$ and an open neighborhood $U_p$ of $t_p$ such that for every $g\in U_p$ the local splitting field of $\chi_g$ over $K_p$ contains a quadratic extension of $K_p$ obtained from a biquadratic extension of $\Q_p$.
\end{lemma}

\begin{proof}
Write $F=\Q_p$ and $K=K_p$.  Since $p\in S_\Delta$, the quadratic algebra $K/F$ is a field.  Choose another quadratic field $M/F$ with $M\ne K$ and set
\[
 L=KM.
\]
Then $L/F$ is biquadratic and $L/K$ is quadratic.  Consider the etale $F$-algebra
\[
 E=M\times F^{d-2},
 \qquad
 L_E=K\otimes_FE=L\times K^{d-2}.
\]
The involution $c$ induced by $K/F$ acts on $L$ and fixes $M$.

For $\lambda\in E^\times$ consider the Hermitian space
\[
 V_{E,\lambda}=(L_E,h_\lambda),
 \qquad
 h_\lambda(x,y)=\Tr_{L_E/K}(\lambda xy^c).
\]
The norm-one torus $L_E^1$ acts unitarily on this space.  Locally, the isometry class of a non-degenerate Hermitian form is determined by its dimension and determinant class, and
\[
 \det V_{E,\lambda}
 =N_{E/F}(\lambda)\det_F(E)
 \quad\text{in }F^\times/N_{K/F}(K^\times).
\]
Because $E$ has split $F$-factors, the norm $N_{E/F}:E^\times\to F^\times$ is surjective.  We may therefore choose $\lambda$ so that $V_{E,\lambda}$ has the same local determinant class as $H_K\otimes_KK_p$.  The local classification of Hermitian forms gives an isometry between them, and transports $L_E^1\cap\SU(V_{E,\lambda})$ to a maximal torus $T<\mathbf G(F)$.

We now choose an element of $T$ whose quadratic block generates $L/K$.  Let $\theta\in K$ satisfy $c(\theta)=-\theta$, choose $m\in M\setminus F$, and put $z=\theta m$.  Then $c(z)=-z$ and $z\notin K$.  For sufficiently small $\varepsilon\in F$ define
\[
 x_0(\varepsilon)=\frac{1+\varepsilon z}{1-\varepsilon z}.
\]
It satisfies
\[
 x_0(\varepsilon)c(x_0(\varepsilon))=1,
\]
tends to $1$, and for nonzero small $\varepsilon$ lies outside $K$.  Since $[L:K]=2$,
\[
 K(x_0(\varepsilon))=L.
\]
Choose the remaining coordinates in $K^1=\{u\in K^\times:u\bar u=1\}$ arbitrarily close to $1$ and adjust one coordinate so that the total determinant is one.  Avoiding finitely many equalities among eigenvalues gives a regular semisimple element $t_p\in T$, and all coordinates may be chosen close enough to $1$ that $t_p\in\Omega$.

The block on $L$ is multiplication by $x_0(\varepsilon)$, so its minimal polynomial over $K$ has degree two and splitting field $L$.  Let $\alpha=x_0(\varepsilon)$.  Since $t_p$ is regular, $\alpha$ is a simple root of its characteristic polynomial.  Krasner's lemma gives a radius $r>0$ such that every root $\beta$ with $|\beta-\alpha|<r$ satisfies $L\subset K(\beta)$.  Roots of separable polynomials depend continuously on the coefficients, and the map $g\mapsto\chi_g$ is continuous.  Shrinking to an open neighborhood $U_p$ of $t_p$, every $g\in U_p$ has a root $\beta$ in that Krasner ball.  Therefore its local splitting field contains $L=K_pM_p$, as required.
\end{proof}

\begin{theorem}\label{thm:open-annihilation}
Let $\Gamma<\mathbf G(\Q)$ be finitely generated and Zariski-dense.  Suppose that the diagonal closure of $\Gamma$ in
\[
 \prod_{p\in S_\Delta}\mathbf G(\Q_p)
\]
is open.  Then there exists a single regular semisimple $\gamma\in\Gamma$ such that
\[
 F=K(E(\gamma))
\]
is a unitary field of definition for $\Gamma$.
\end{theorem}

\begin{proof}
If $S_\Delta=\varnothing$, the finite Landherr obstruction is already trivial.  Since the regular semisimple locus of $\mathbf G$ is a nonempty Zariski-open set and $\Gamma$ is Zariski-dense, choose any regular semisimple $\gamma\in\Gamma$; then $H_K\sim_KJ$, hence also $H_K\sim_{K(E(\gamma))}J$.  Assume from now on that $S_\Delta\ne\varnothing$.

Let
\[
 H=\overline{\Gamma}\subset\prod_{p\in S_\Delta}\mathbf G(\Q_p).
\]
Since $H$ is open and contains the identity, there are identity neighborhoods $\Omega_p\subset\mathbf G(\Q_p)$ such that
\[
 \prod_{p\in S_\Delta}\Omega_p\subset H.
\]
Apply Lemma~\ref{lem:biquad} inside each $\Omega_p$.  It gives a regular element $t_p\in\Omega_p$ and an open neighborhood $U'_p$ of $t_p$ having the stated spectral property.  Let $\mathbf G^{\rm rs}$ denote the regular semisimple locus and replace $U'_p$ by
\[
 U_p=U'_p\cap\Omega_p\cap\mathbf G^{\rm rs}(\Q_p).
\]
This is still a nonempty open neighborhood of $t_p$, and
\[
 \prod_{p\in S_\Delta}U_p\subset H.
\]
Because $\Gamma$ is dense in $H$, this product meets $\Gamma$; choose
\[
 \gamma\in\Gamma\cap\prod_{p\in S_\Delta}U_p.
\]
Since $S_\Delta\ne\varnothing$ and $\gamma$ is regular semisimple in every displayed local factor, its characteristic polynomial has nonzero discriminant over $K$, so $\gamma$ is regular semisimple globally.

Put $F=K(E(\gamma))$.  The full spectrum of a unitary element is stable under $\lambda\mapsto\bar\lambda^{-1}$, hence the splitting field $F$ is stable under complex conjugation; let $F_0=F^{\bar{\ }}$.  Fix $p\in S_\Delta$ and a place $w$ of $F$ above $p$, and let $v$ be the place of $F_0$ below it.  Lemma~\ref{lem:biquad} supplies a quadratic extension of $K_p$ inside $F_w$, namely
\[
 K_pM_p\subset F_w,
 \qquad [K_pM_p:K_p]=2.
\]
Consequently $[F_w:K_p]$ is even.  If $F/F_0$ splits at $v$, the local norm map is surjective and there is no obstruction.  Otherwise $w$ is fixed by conjugation and
\[
 [F_w:F_{0,v}]=2.
\]
Since $p\in S_\Delta$, the quadratic algebra $K_p/\Q_p$ is a field.  Conjugation acts nontrivially on $K_p$ and trivially on $F_{0,v}$, so
\[
 K_p\cap F_{0,v}=\Q_p.
\]
Moreover $F=KF_0$, and in the present non-split local case this gives
\[
 F_w=K_pF_{0,v}.
\]
Comparing degrees in this compositum yields
\[
 [F_{0,v}:\Q_p]=[F_w:K_p],
\]
which is even.  Restriction of the local Brauer invariant $1/2$ therefore multiplies it by an even degree and kills the Landherr obstruction.  Outside $S_\Delta$ the obstruction was already trivial, and the archimedean signatures agree.  Landherr's theorem gives $H_K\sim_FJ$.
\end{proof}

Pink's openness results provide criteria for the hypothesis of Theorem~\ref{thm:open-annihilation} \cite{Pink1998,Pink2000}.  After simultaneous openness is controlled, only the discrete local case remains.

\subsection{Simultaneous local openness}\label{subsec:pink-en}

We record a criterion for simultaneous openness.  Let $\mathbf G_{\rm ad}$ be the adjoint quotient of $\mathbf G$, and let $\overline\Gamma$ be the image of $\Gamma$ in $\mathbf G_{\rm ad}(\Q)$.  Pink formulates the weak-approximation theorem after passage to a minimal quasi-model.  In the present characteristic-zero $\Q$-setting this causes no extra hypothesis: $\Q$ has no proper global subfield, and there are no non-standard positive-characteristic isogenies, so the adjoint $\Q$-form with the Zariski-dense group $\overline\Gamma$ is already in the required minimal situation.  Apply Pink's finite-place weak-approximation decomposition \cite[Theorems~6.12 and~6.15]{Pink2000} to the finite set
\[
 V=S_\Delta.
\]
Pink produces a partition $V=V_o\sqcup V_d$ such that the closure of the projection of $\overline\Gamma$ to the $V_d$-factors is discrete, while the closure of its generalized commutator group in the simply connected cover is open on the $V_o$-factors.

If $V_d$ were nonempty, the projection of $\Gamma$ to
\[
 \prod_{p\in V_d}\mathbf G(\Q_p)
\]
would also be discrete: the central covering $\mathbf G\to\mathbf G_{\rm ad}$ has finite fibers, so discreteness is preserved under pullback to the image of $\Gamma$.  This contradicts the hypothesis below applied to the nonempty subset $V_d$.  Hence
\[
 V_d=\varnothing,
 \qquad V_o=S_\Delta.
\]
Pink's Theorem~6.12(c) therefore gives an open closure of the generalized commutator group in
\[
 \prod_{p\in S_\Delta}\mathbf G(\Q_p),
\]
because $\mathbf G$ is the simply connected cover of $\mathbf G_{\rm ad}$.  If $\bar g,\bar h\in\overline\Gamma$ are lifted by $g,h\in\Gamma$, their generalized commutator is the ordinary commutator $[g,h]$; central changes of the lifts do not alter it.  Thus the generalized commutator group is contained in $[\Gamma,\Gamma]\subset\Gamma$.  Its open closure is consequently contained in the diagonal closure of $\Gamma$, so that diagonal closure is itself open.  This is exactly the hypothesis needed in Theorem~\ref{thm:open-annihilation}.

\begin{proposition}\label{prop:pink-obstructive-en}
Assume that for every nonempty $T\subset S_\Delta$ the closure of the diagonal projection of $\Gamma$ to $\prod_{p\in T}\mathbf G(\Q_p)$ is non-discrete.  Then a single regular semisimple element $\gamma\in\Gamma$ has a spectral field that kills the full finite Landherr obstruction.
\end{proposition}

\begin{proof}
The preceding paragraph gives openness of the full diagonal closure.  Theorem~\ref{thm:open-annihilation} then applies.
\end{proof}

There is an especially useful automatic source of non-discreteness.

\begin{lemma}\label{lem:unipotent-small-en}
Let $u\in\GL_d(K)$ be a nontrivial unipotent element and let $p_1,\ldots,p_s$ be rational primes.  Put
\[
 m_r=(p_1\cdots p_s)^r.
\]
Then
\[
 u^{m_r}\longrightarrow I
\]
simultaneously in all the $p_j$-adic topologies.
\end{lemma}

\begin{proof}
Write $u=I+N$ with $N$ nilpotent.  The binomial expansion is finite:
\[
 u^m=I+mN+\binom m2N^2+\cdots.
\]
For each fixed exponent $k$, the $p_j$-adic valuation of $\binom{m_r}{k}$ tends to $+\infty$ as $r\to\infty$.  Hence every nonconstant coefficient tends to zero simultaneously.
\end{proof}

Consequently, the presence of a nontrivial unipotent rules out discreteness in every obstructive completion at once.  The difficult residual regime is therefore naturally sought among totally loxodromic or Schottky-type groups.

\section{A discrete Schottky model}\label{sec:schottky}

Local openness does not address discrete closures.  The following family shows that this branch is genuine, while at the same time remaining explicit enough for the spectral obstruction to be computed directly.  Consider the basic form
\[
 H_3=\mathbb H\oplus\langle1,3\rangle.
\]
Set
\[
 A_s=\diag\left(3,i,-i,\frac13\right)
\]
and
\[
 C_s=
 \begin{pmatrix}
 -4&5-5i&-3+3i&7\\
 1+i&-3&0&-2-2i\\
 1+i&-2&1&-1-i\\
 1&-1+i&0&-1
 \end{pmatrix},
 \qquad
 B_s=C_sA_sC_s^{-1}.
\]
One checks exactly that
\[
 C_s^*H_3C_s=H_3,\qquad \det C_s=1.
\]

\begin{proposition}\label{prop:schottky}
There exists $N_0$ such that for every sufficiently large odd $N$,
\[
 \Gamma_N=\langle A_s^N,B_s^N\rangle
\]
is free of rank two and discrete both in the complex topology and in $\SU(H_3,\Q_3(i))$.  Every nontrivial element is biproximal in the corresponding projective actions, and $\Gamma_N$ is Zariski-dense in the ambient special unitary group.
\end{proposition}

\begin{proof}
At the complex place, $A_s$ is biproximal with attracting and repelling lines $[e_1]$ and $[e_4]$.  At the $3$-adic place the same two lines occur with their dynamical roles reversed because
\[
 |3|_3<1<|1/3|_3.
\]
The corresponding data for $B_s$ are obtained by applying $C_s$.  Simultaneous ping--pong reduces to the non-vanishing of the four relevant corner entries of $C_s$ and $C_s^{-1}$.  Here
\[
 (C_{11},C_{41},C_{14},C_{44})=(-4,1,7,-1)
\]
and
\[
 ((C_s^{-1})_{11},(C_s^{-1})_{41},(C_s^{-1})_{14},(C_s^{-1})_{44})=(-1,1,7,-4),
\]
so all required transversality conditions hold.  A common sufficiently large power therefore gives proximal ping--pong at both places.  This proves freeness, discreteness, and biproximality of every nontrivial word.

For Zariski density, let $\mathbf H$ be the Zariski closure of $\langle A_s,B_s\rangle$.  The identity component of the closure of $\langle A_s\rangle$ contains
\[
 t\longmapsto\diag(t,1,1,t^{-1}),
\]
so its Lie algebra contains
\[
 X=\diag(1,0,0,-1).
\]
Similarly it contains $Y=C_sXC_s^{-1}$.  Since the Lie algebra of $\mathbf H$ is stable under brackets and under the adjoint actions of $A_s$ and $B_s$, it contains the following fifteen elements:
\[
\begin{gathered}
 Z_1=X,\quad Z_2=Y,\quad Z_3=\Ad(B_s)X,\quad Z_4=\Ad(A_s)Y,\\
 Z_5=[X,Y],\quad Z_6=[X,Z_3],\quad Z_7=[X,Z_4],\\
 Z_8=[Y,Z_3],\quad Z_9=[Y,Z_4],\quad Z_{10}=[Y,Z_5],\\
 Z_{11}=[Y,Z_6],\quad Z_{12}=[Y,Z_7],\quad Z_{13}=[Z_3,Z_4],\\
 Z_{14}=[Z_3,Z_5],\quad Z_{15}=[Z_3,Z_6].
\end{gathered}
\]
Reduce these matrices modulo $5$ using $i\mapsto2\in\mathbb F_5$; the denominators occurring in $A_s$ are invertible modulo $5$.  In the coordinate system consisting of the twelve off-diagonal entries followed by the first three diagonal coordinates of $\mathfrak{sl}_4$, the $15\times15$ matrix with columns $Z_1,\ldots,Z_{15}$ has determinant
\[
 1\pmod5.
\]
Hence the corresponding determinant in characteristic zero is nonzero.  The $Z_j$ are therefore linearly independent over $\Q(i)$ and span $\mathfrak{sl}_4$.  Thus $\mathbf H$ has complexified Lie algebra $\mathfrak{sl}_4$ and equals the ambient special unitary group in the relevant Zariski sense.  The exact finite-field computation is recorded in Appendix~\ref{app:schottky-certificate}.

Finally, for odd $N$ the power map preserves the components of the algebraic closure of $\langle A_s\rangle$, whose finite part has order four.  Therefore $\langle A_s^N,B_s^N\rangle$ has the same Zariski closure as $\langle A_s,B_s\rangle$.
\end{proof}

The family shows that local discreteness cannot be eliminated by a general openness argument.  Nevertheless, the Hermitian obstruction can still be killed by a single spectral field if the local splitting degrees are even.

\begin{proposition}\label{prop:even-degree}
Let $K/\Q$ be imaginary quadratic, let $H_K$ have even matrix dimension, and let $\Delta$ represent its finite Landherr obstruction relative to $J$.  Let $\gamma\in\SU(H_K,K)$ be regular semisimple and put
\[
 F=K(E(\gamma)).
\]
Then $F/K$ is the splitting field of $\chi_\gamma$, hence is Galois, and unitarity makes $F$ stable under complex conjugation.  Let $F_0=F^{\bar{\ }}$.  Suppose that for every $p\in S_\Delta$ and one (equivalently every) completion $F_w$ above the unique place $K_p$ of $K$ over $p$,
\[
 [F_w:K_p]
\]
is even.  Then
\[
 H_K\sim_FJ.
\]
\end{proposition}

\begin{proof}
Fix $p\in S_\Delta$ and let $v$ be the place of $F_0$ below $w$.  The relevant local Landherr obstruction is represented by a Brauer invariant $1/2$ over $\Q_p$.  If $F/F_0$ splits at $v$, the local norm map is surjective and the obstruction vanishes automatically.  Otherwise
\[
 [F_w:F_{0,v}]=2.
\]
Since $p\in S_\Delta$, the quadratic algebra $K_p/\Q_p$ is a field, so also $[K_p:\Q_p]=2$.  Because conjugation acts nontrivially on $K_p$ and trivially on $F_{0,v}$,
\[
 K_p\cap F_{0,v}=\Q_p.
\]
Also $F=KF_0$ globally, and in the present non-split local case this gives
\[
 F_w=K_pF_{0,v}.
\]
Thus the two degree formulas inside $F_w$ give
\[
 [F_{0,v}:\Q_p]=[F_w:K_p].
\]
By hypothesis this degree is even.  Restriction of the Brauer class from $\Q_p$ to $F_{0,v}$ multiplies its invariant by this degree, so $1/2$ becomes zero.  Hence the finite obstruction vanishes at every place over $S_\Delta$.  Outside $S_\Delta$ it was already trivial, while the archimedean signatures agree.  Landherr's local--global theorem gives $H_K\sim_FJ$.
\end{proof}

\begin{proposition}\label{prop:schottky-kill}
For odd $N$, let
\[
 \gamma_N=A_s^NB_s^N,
 \qquad F_N=\Q(i)(E(\gamma_N)).
\]
Then
\[
 H_3\sim_{F_N}J.
\]
Thus, for $N\ge N_0$, the Zariski-dense group $\Gamma_N$ is not a counterexample to the total spectral-field conjecture; a word of length two already kills the norm class $[3]$.
\end{proposition}

\begin{proof}
Put $u=3^N$ and $\varepsilon=i^N$.  Since $N$ is odd,
\[
 \varepsilon^2=-1,
 \qquad \bar\varepsilon=-\varepsilon.
\]
Direct multiplication gives
\[
 \chi_N(X)=X^4-T_NX^3+S_NX^2-\overline{T_N}X+1,
\]
with
\[
 T_N=\frac4{u^2}(u^4+4u^2+1-2\varepsilon u^3-2\varepsilon u),
 \qquad
 S_N=\frac2{u^2}(u^2-2)(2u^2-1).
\]
At $p=3$ one obtains
\[
 \operatorname{Disc}(\chi_N)
 =-\frac{2^{12}(u^2+1)^2}{u^{12}}R(u),
\]
where
\begin{align*}
 R(u)={}&3u^{20}+112u^{18}+1608u^{16}+10976u^{14}+32760u^{12}\\
 &+49374u^{10}+32760u^8+10976u^6+1608u^4+112u^2+3.
\end{align*}
Because $u=3^N$, one has $R(u)\equiv3\pmod9$, and therefore
\[
 v_3(\operatorname{Disc}(\chi_N))=1-12N,
\]
which is odd.  The discriminant is not a square in $K_3=\Q_3(i)$, so the local splitting field contains a quadratic subextension and has even degree.

At $p=2$, let $\pi=1+i$ and set
\[
 Q_N(Z)=\pi^{-8}\chi_N(1+\pi^2Z).
\]
The explicit formulas for $T_N,S_N$, together with $u^2\equiv1\pmod8$ and $\varepsilon\equiv1\pmod\pi$, show that $Q_N(Z)$ is monic integral and
\[
 \overline{Q_N}(Z)=Z^4+Z^2+1=(Z^2+Z+1)^2
\]
in $\mathbb F_2[Z]$.  Let $M/K_2$ be a splitting field and let $z\in M$ be any root of $Q_N$.  Since $Q_N$ is monic integral, $z$ is integral.  Reducing its equation in the residue field $k_M$ gives
\[
 (\bar z^2+\bar z+1)^2=0.
\]
Because $k_M$ is a field, $\bar z^2+\bar z+1=0$.  The polynomial $Z^2+Z+1$ is irreducible over $\mathbb F_2$, so $\mathbb F_4\subset k_M$.  Hence the residue degree, and therefore $[M:K_2]$, is even.  Consequently
\[
 [K_2(E(\gamma_N)):K_2]\equiv0\pmod2.
\]
The norm class $[3]$ is obstructive only at $2$ and $3$.  Proposition~\ref{prop:even-degree} therefore yields $H_3\sim_{F_N}J$.
\end{proof}

\section{Generic spectral fields and local parity}\label{sec:S4}

The Schottky calculation is explicit and special.  Generic elements provide a complementary view: after linear descent the Zariski closure is $\SL_4$, so the Weyl group controls the typical splitting behavior of characteristic polynomials.

\begin{proposition}\label{prop:genericS4}
Let $K$ be a number field and let $\Gamma<\SL_4(K)$ be finitely generated and Zariski-dense.  Then $\Gamma$ contains a regular semisimple element $\gamma$ of infinite order such that
\[
 \Gal(K(E(\gamma))/K)\simeq S_4.
\]
Moreover, generic elements form a nonempty open subset in the profinite topology of $\Gamma$.
\end{proposition}

\begin{proof}
After replacing $\Gamma$ by a finite-index torsion-free subgroup, which remains Zariski-dense because $\SL_4$ is connected, the generic-element theorem of Prasad--Rapinchuk supplies regular elements whose connected centralizer is a generic maximal torus and shows that such elements form a nonempty profinite-open subset \cite{PrasadRapinchuk2017}.  For type $A_3$, the Weyl group is
\[
 W(A_3)\simeq S_4.
\]
A generic maximal torus therefore has Galois action containing the full $S_4$ on its four weights.  Equivalently, the associated etale algebra of degree four is a field whose normal closure has Galois group $S_4$.

Choose such a generic element $\gamma$ of infinite order and let $T=Z_{\SL_4}(\gamma)^\circ$.  The identity component of the Zariski closure of $\langle\gamma\rangle$ is a $K$-subtorus of $T$.  For a generic torus of type $A_3$, the Galois action on
\[
 X^*(T)\otimes_\mathbb Z\Q\simeq\{(x_1,x_2,x_3,x_4)\in\Q^4:\ x_1+x_2+x_3+x_4=0\}
\]
contains the standard $S_4$-representation, which is irreducible over $\Q$.  Hence $T$ has no nontrivial proper $K$-subtorus.  Since $\gamma$ has infinite order, its Zariski closure has positive-dimensional identity component, so that component is all of $T$; thus $\langle\gamma\rangle$ is Zariski-dense in $T$.

The four weights of the standard representation generate $X^*(T)$.  If an automorphism of the splitting field of $T$ fixes all four eigenvalues of $\gamma$, then for every standard weight $\chi_i$ one has
\[
 \sigma(\chi_i)(\gamma)=\chi_i(\gamma).
\]
Zariski density of $\langle\gamma\rangle$ in $T$ forces $\sigma(\chi_i)=\chi_i$ for all $i$, hence $\sigma$ acts trivially on $X^*(T)$.  Therefore $K(E(\gamma))$ is the splitting field of $T$ itself, and
\[
 \Gal(K(E(\gamma))/K)\simeq S_4.
\]
The relation with splitting fields of generic elements is also discussed in \cite{PisolkarRajan2016}.
\end{proof}

\begin{proposition}\label{prop:S4-sieve}
Let $\mathfrak p$ be an obstructive finite place of $K$ and let $\gamma$ be as in Proposition~\ref{prop:genericS4}.  If the local Landherr class remains nontrivial over $K(E(\gamma))$, then the decomposition group at $\mathfrak p$ has odd order.  Hence
\[
 D_{\mathfrak p}=1
 \quad\text{or}\quad
 D_{\mathfrak p}\simeq C_3.
\]
Equivalently, locally the characteristic polynomial has factorization type
\[
 1+1+1+1
 \quad\text{or}\quad
 1+3.
\]
\end{proposition}

\begin{proof}
Let $L=K(E(\gamma))$.  The local Landherr obstruction has order two, represented by Brauer invariant $1/2$.  Restriction to a local extension of even degree would kill this invariant.  Hence persistence of the obstruction over a completion $L_w/K_{\mathfrak p}$ forces
\[
 [L_w:K_{\mathfrak p}]
\]
to be odd.  Since $L/K$ is Galois with group $S_4$, this local degree is the order of the decomposition group $D_{\mathfrak p}$.  The only odd-order subgroups of $S_4$ are the trivial group and a cyclic group generated by a $3$-cycle.  Their orbit decompositions on the four roots are respectively $1+1+1+1$ and $1+3$.
\end{proof}

\begin{corollary}\label{cor:frontier}
If the total spectral field fails to kill the Hermitian obstruction, then there is an obstructive place $\mathfrak p$ such that every local spectral field $K_{\mathfrak p}(E(\gamma))$, $\gamma\in\Gamma$, has odd degree.  In particular, every generic $S_4$ element has local decomposition group $1$ or $C_3$ at that same place.
\end{corollary}

\begin{proof}
Suppose instead that for every obstructive place $\mathfrak p$ there were some $\gamma_{\mathfrak p}\in\Gamma$ whose local splitting field had even degree.  There are only finitely many obstructive places, so the compositum
\[
 M=K\bigl(E(\gamma_{\mathfrak p}):\mathfrak p\in S_\Delta\bigr)
\subset F_{\rm eig}(\Gamma)
\]
is a finite Galois extension stable under complex conjugation.  At each obstructive place a completion of $M$ contains the chosen even-degree local splitting field, hence has even degree over the corresponding completion of $K$; by Galoisity the same holds for every completion above that place.  The local degree argument of Proposition~\ref{prop:even-degree}, equivalently the finite-compositum argument used later in Lemma~\ref{lem:compositum}, then kills every finite Landherr obstruction over $M$.  This contradicts the assumed failure over the larger field $F_{\rm eig}(\Gamma)$.  Thus one obstructive place has odd local spectral degree for every $\gamma\in\Gamma$, and Proposition~\ref{prop:S4-sieve} gives the final assertion for generic $S_4$ elements.
\end{proof}

Thus a residual counterexample must display a striking tension: globally, generic elements have full $S_4$ Galois group, but at one fixed obstructive place every one of those elements lies in one of two odd local strata.

\subsection{Quartic parity}\label{subsec:quartic-parity-en}

Let $f\in F[X]$ be a separable quartic over a local field and let $L/F$ be its splitting field.  The following elementary reformulation will be useful.

\begin{proposition}\label{prop:quartic-odd-en}
The degree $[L:F]$ is odd if and only if the Galois group of $f$ is trivial or cyclic of order three.  Equivalently, the factorization type of $f$ over $F$ is either
\[
 1+1+1+1
\]
or
\[
 1+3
\]
with the cubic factor cyclic.  Every other factorization or Galois type forces an even splitting degree.
\end{proposition}

\begin{proof}
The Galois group embeds in $S_4$ and its order equals $[L:F]$.  The only subgroups of odd order in $S_4$ are the trivial group and cyclic groups generated by a $3$-cycle.  Their orbit decompositions on four letters are exactly the two types displayed above.
\end{proof}

\subsection{Trace and resolvent criteria}

For $g\in\SU(3,1)$ write
\[
 \chi_g(X)=X^4-t_gX^3+\sigma_gX^2-\overline{t_g}X+1,
\]
where
\[
 t_g=\tr(g),
 \qquad
 \sigma_g=\frac{t_g^2-\tr(g^2)}2.
\]
Thus two traces determine the characteristic polynomial.  A convenient cubic resolvent is
\[
 \mathcal R_g(Y)
 =Y^3-\sigma_gY^2+(t_g\bar t_g-4)Y
 +(4\sigma_g-t_g^2-\bar t_g^2).
\]
The quartic has odd splitting degree precisely when it is completely split or when it has factorization type $1+3$ with an irreducible cyclic cubic factor.  The cubic resolvent is useful only after this factorization information is retained.  If the quartic is completely split, then $\mathcal R_g$ is completely split.  If the quartic has Galois group $C_3$ and factorization type $1+3$, the induced action on the three partitions of four roots into two pairs is a $3$-cycle; hence $\mathcal R_g$ is an irreducible cyclic cubic.

Conversely, \emph{irreducibility of the cubic resolvent together with square quartic discriminant is not by itself a parity criterion}: it is also compatible with Galois group $A_4$ when the quartic is irreducible.  On the separable locus, the nontrivial odd case is therefore detected by the conjunction
\[
 \chi_g\text{ has an }F\text{-rational root},\qquad
 \chi_g/(X-r)\text{ is irreducible cubic},\qquad
 \disc(\chi_g)\in F^{\times2}.
\]
Equivalently, after the $1+3$ factorization has been verified, the irreducible cubic factor is cyclic.  Since $\chi_g$ and $\mathcal R_g$ are determined by $\tr(g)$ and $\tr(g^2)$, the pair
\[
 \bigl(\chi_g,\mathcal R_g\bigr)
\]
still provides an effective local parity test, but the factorization type of $\chi_g$ is an essential part of the test.

The parity condition is also stable on the regular locus.

\begin{proposition}\label{prop:krasner-en}
Let $\mathcal R$ be an algebraic family of representations over a non-archimedean local field $F$, let $w$ be a fixed word, and restrict to the open locus on which $\chi_{\rho(w)}$ is separable.  Then the parity
\[
 [F(E(\rho(w))):F]\pmod2
\]
is locally constant on a sufficiently small analytic neighborhood of every point.
\end{proposition}

\begin{proof}
Fix a base point and write its separable characteristic polynomial as
\[
 f_0(X)=\prod_{i=1}^d(X-\alpha_i)
\]
in a fixed algebraic closure, with splitting field $L/F$.  Choose pairwise disjoint sufficiently small balls $B_i$ around the $\alpha_i$ so that Krasner's lemma applies: whenever $\beta_i\in B_i$, one has $F(\alpha_i)\subset F(\beta_i)$.  After shrinking the analytic parameter neighborhood, every nearby characteristic polynomial has exactly one root $\beta_i$ in each $B_i$.

Because the coefficients of the perturbed polynomial are sufficiently close to those of $f_0$, Hensel's lemma applied inside the complete field $L$ produces the unique root in each $B_i$ already in $L$.  Hence the nearby splitting field $L'$ satisfies $L'\subset L$.  Conversely, Krasner gives $F(\alpha_i)\subset F(\beta_i)\subset L'$ for every $i$, so $L\subset L'$.  Thus $L'=L$ throughout a sufficiently small neighborhood.  In particular the splitting degree, and hence its parity, is locally constant.
\end{proof}

A second point is equally important: the odd stratum is not excluded merely by the Hermitian isometry class.  Maximal unitary tori of type $1+3$ can be embedded in either local Hermitian class in dimension four by adjusting the admissible embedding parameter.  Thus both odd and even spectral strata occur inside each ambient local unitary group.  The residual problem is therefore a distribution problem for the subgroup, not a formal restriction imposed by the Hermitian form itself.

\begin{proposition}\label{prop:local-C3-en}
Let $F/F_0$ be a quadratic extension of non-archimedean local fields and let $V$ be any four-dimensional non-degenerate Hermitian space over $F/F_0$.  Then $\SU(V)(F_0)$ contains a regular semisimple element whose characteristic polynomial over $F$ has factorization type $1+3$ and whose cubic factor has cyclic splitting field.
\end{proposition}

\begin{proof}[Construction]
Choose a cyclic cubic extension $E_3/F_0$ and set $E=F_0\times E_3$.  Then
\[
 L_E=F\otimes_{F_0}E\simeq F\times(FE_3).
\]
The norm-one torus of $L_E/E$ has one one-dimensional factor and one three-dimensional Galois orbit.  The admissible embeddings of this abstract unitary torus into four-dimensional Hermitian spaces are obtained by varying the parameter $\lambda\in E^\times$ in the trace form
\[
 h_\lambda(x,y)=\Tr_{L_E/F}(\lambda xy^c).
\]
Its determinant class is multiplied by $N_{E/F_0}(\lambda)$.  Since the split factor of $E$ makes this norm map surjective, one can choose $\lambda$ so that $h_\lambda$ has the prescribed local Hermitian class $V$.

It remains to impose the special-unitary determinant condition.  Choose
\[
 x_3\in(FE_3)^\times,
 \qquad x_3c(x_3)=1,
\]
generic enough that its orbit under $\Gal(FE_3/F)$ has size three, and set
\[
 x_0=N_{FE_3/F}(x_3)^{-1}\in F^\times.
\]
Because $c(x_3)=x_3^{-1}$, one has $x_0c(x_0)=1$, while
\[
 x_0N_{FE_3/F}(x_3)=1.
\]
Thus multiplication by $x=(x_0,x_3)$ lies in the determinant-one norm torus $L_E^1\cap\SU(V_{E,\lambda})$.  Avoiding the proper algebraic subsets on which $x_3$ has smaller Galois orbit or $x_0$ coincides with one of its three conjugates, we obtain a regular semisimple element.  Its characteristic polynomial has one linear factor over $F$ and an irreducible cubic factor whose splitting field is the cyclic cubic extension $FE_3/F$, giving the required $1+3$ factorization.
\end{proof}

By contrast, the biquadratic construction of Lemma~\ref{lem:biquad} supplies an open even-degree stratum.  Hence both parities occur on nonempty analytic open sets after suitable choice of regular torus elements.  This observation is one reason the rank-two problem cannot be solved by the local classification of Hermitian forms alone.

\subsection{Odd local degree as a Diophantine condition}

Corvaja proves rigidity results for Zariski-dense linear group actions under rationality assumptions on fixed points \cite{Corvaja2007}.  They suggest a possible diophantine route to controlling spectral data in a fixed finite extension, although a reduction to Corvaja's hypotheses is not presently available.

The preceding conclusion is local: a counterexample to total spectral descent must have an obstructive place at which every local spectral field has odd degree.  A global escape statement would require an additional diophantine argument.  Fields of totally $S$-adic numbers \cite{BarySorokerFehm2014} provide one natural setting for such a refinement.

\section{The rank-one local problem}\label{sec:rank1}

We turn to the discrete local branch.  When the Hermitian form is non-quasi-split in dimension four, its Witt index is one and the building is a tree.  More importantly, a hyperbolic element carries a distinguished two-dimensional transverse Hermitian factor.  This is the feature that makes the rank-one case accessible.

Let $F/F_0$ be a quadratic extension of non-archimedean local fields and let $V$ be a four-dimensional Hermitian space.  Assume throughout this section that $V$ is \emph{non-quasi-split}, equivalently that its Witt index is one.  Then $G=\SU(V)(F_0)$ has relative rank one and its Bruhat--Tits building is a tree.

\begin{proposition}\label{prop:rotation}
Let $g\in G$ be semisimple with non-relatively-compact cyclic subgroup.  Then there is a Hermitian decomposition
\[
 V=\ell_+(g)\oplus W_g\oplus\ell_-(g),
 \qquad \dim_F\ell_\pm(g)=1,\quad\dim_FW_g=2,
\]
with $W_g$ anisotropic, and an adapted basis in which
\[
 g=\begin{pmatrix}
 \alpha_g&0&0\\
 0&R_g&0\\
 0&0&\bar\alpha_g^{-1}
 \end{pmatrix},
 \qquad R_g\in U(W_g)(F_0).
\]
Hence
\[
 \chi_g(X)=
 (X-\alpha_g)(X-\bar\alpha_g^{-1})q_g(X),
\]
where
\[
 q_g(X)=\det(XI_{W_g}-R_g)\in F[X]
\]
is quadratic.  The local splitting field of $\chi_g$ over $F$ therefore has degree $1$ or $2$.
\end{proposition}

\begin{proof}
Let $T_g$ be the algebraic closure of the semisimple part of $\langle g\rangle$ and let $S_g$ be its maximal $F_0$-split subtorus.  Since $\langle g\rangle$ is not relatively compact, $S_g$ is nontrivial.  Relative rank one forces $S_g$ to be a maximal split torus.  The standard representation decomposes into its three $S_g$-weight spaces
\[
 V=V_+\oplus V_0\oplus V_-,
 \qquad \dim_FV_+=\dim_FV_-=1,
 \quad \dim_FV_0=2.
\]
The Hermitian form pairs $V_+$ and $V_-$ and restricts anisotropically to $V_0$; set $\ell_\pm(g)=V_\pm$ and $W_g=V_0$.  Because $g$ centralizes $S_g$, it preserves this decomposition.  The unitary relation identifies the two extreme eigenvalues as $\alpha_g$ and $\bar\alpha_g^{-1}$, while the middle block belongs to $U(W_g)(F_0)$.  The unitary group of an anisotropic Hermitian plane is compact, so its two eigenvalues are units.  The displayed factorization follows, and only the quadratic factor may require an extension of $F$.
\end{proof}

\begin{definition}\label{def:rot-splitting}
For $g$ as above, the primary invariant is the splitting status of the transverse quadratic polynomial $q_g$ over $F$.

If $q_g$ is separable, define its rotational discriminant by
\[
 \delta_{\rm rot}(g)=
 [\disc(q_g)]\in F^\times/F^{\times2}.
\]
If $q_g$ has a repeated root, then $\disc(q_g)=0$ and \emph{no element of $F^\times/F^{\times2}$ is assigned to it}.  For uniform bookkeeping we also use the extended symbol
\[
 \widetilde\delta_{\rm rot}(g)=
 \begin{cases}
 0,&\disc(q_g)=0,\\
 [\disc(q_g)],&\disc(q_g)\ne0,
 \end{cases}
 \qquad
 \widetilde\delta_{\rm rot}(g)\in
 \{0\}\sqcup F^\times/F^{\times2}.
\]
In characteristic zero a quadratic polynomial $q_g$ splits over $F$ if and only if
\[
 \widetilde\delta_{\rm rot}(g)\in\{0,1\},
\]
and it is irreducible over $F$ if and only if
\[
 \widetilde\delta_{\rm rot}(g)\in
 F^\times/F^{\times2}\setminus\{1\}.
\]
Thus the square-class discriminant is used only on the separable locus, while the repeated-root split case is recorded by the separate value $0$.
\end{definition}

\begin{corollary}\label{cor:rank1-parity}
Let $\Lambda<G$ be discrete and torsion-free.  Then every nontrivial $g\in\Lambda$ is hyperbolic in the tree and algebraically semisimple.  Moreover
\[
 [F(E(g)):F]=1
 \quad\Longleftrightarrow\quad
 q_g\text{ splits over }F
 \quad\Longleftrightarrow\quad
 \widetilde\delta_{\rm rot}(g)\in\{0,1\},
\]
whereas
\[
 [F(E(g)):F]=2
 \quad\Longleftrightarrow\quad
 q_g\text{ is irreducible over }F.
\]
On the separable locus this reduces to the familiar criterion $[F(E(g)):F]=1$ if and only if $\delta_{\rm rot}(g)=1$.
\end{corollary}

\begin{proof}
The Bruhat--Tits action is without inversions after the standard barycentric subdivision.  Thus a nontrivial element of $\Lambda$ is either elliptic or hyperbolic.  If it were elliptic, it would fix a point of the tree and hence lie in a compact point stabilizer.  The intersection of a compact subgroup with the discrete group $\Lambda$ is finite, and torsion-freeness would force $g=e$.  Therefore every nontrivial $g$ is hyperbolic.

Such a $g$ fixes two opposite boundary points.  Hence it belongs to the intersection of the corresponding opposite minimal parabolics, namely the Levi subgroup
\[
 L_g=Z_G(S_g).
\]
In relative rank one, $L_g$ is generated by the split torus $S_g$ and an anisotropic reductive factor.  Therefore $L_g(F_0)$ contains no nontrivial $F_0$-rational unipotent element.  In the Jordan decomposition of $g$, the unipotent part also lies in $L_g(F_0)$, so it is trivial; thus $g$ is algebraically semisimple.  Proposition~\ref{prop:rotation} applies.  The two extreme eigenlines are $F$-rational isotropic lines, so their eigenvalues lie in $F$, and the only possible extension needed to split $\chi_g$ is the quadratic transverse factor $q_g$.  The stated degree and discriminant criteria now follow from Definition~\ref{def:rot-splitting}.
\end{proof}

The elementary rival is a Schottky group embedded in the hyperbolic plane factor
\[
 \SU(\mathbb H)(F_0)\hookrightarrow\SU(\mathbb H\oplus W)(F_0),
\]
acting trivially on the anisotropic plane $W$.  Its rotational polynomial is $(X-1)^2$ for every element, but its Zariski closure is proper.  Thus the ``all rotations split'' condition is compatible with discreteness and freeness, but not with Zariski density in this elementary model.  This is exactly why the later bending and holonomy arguments are needed.

\begin{example}\label{ex:rank-one-rival-en}
Let $\Lambda_0<\SU(\mathbb H)(F_0)\simeq\SL_2(F_0)$ be a torsion-free Schottky group.  Embed it by
\[
 \gamma\longmapsto \gamma\oplus I_W.
\]
Every nontrivial element is hyperbolic in the Bruhat--Tits tree, but
\[
 q_g(X)=(X-1)^2,
 \qquad \widetilde\delta_{\rm rot}(g)=0
\]
for all $g\in\Lambda_0$.  In particular every rotational factor is split, although the ordinary square-class discriminant $\delta_{\rm rot}(g)$ is undefined because the discriminant is zero.  Its Zariski closure is contained in
\[
 S\bigl(U(\mathbb H)\times U(W)\bigr),
\]
a proper algebraic subgroup.
\end{example}

The local rigidity question asks whether the condition ``$q_g$ splits over $F$ for every nontrivial word'' can persist after imposing Zariski density.

\section{Bending and rotational splitting}\label{sec:bending}

Before treating general holonomy, it is useful to see the transverse obstruction move in an explicit deformation.  A small bending of the rank-one model produces a quadratic rotational factor whose discriminant can be read to leading order.

Let $F/F_0$ be a quadratic extension of non-archimedean local fields of characteristic zero, let
\[
 H_q=\begin{pmatrix}
 0&0&0&1\\
 0&1&0&0\\
 0&0&q&0\\
 1&0&0&0
 \end{pmatrix},
 \qquad q\in F_0^\times,
\]
and assume that $\langle1,q\rangle$ is anisotropic.  Choose
\[
 a\in F_0^\times,
 \qquad |a|\ne1,
 \qquad
 A=\diag(a,1,1,a^{-1}).
\]
For parameters
\[
 r,s,t,u\in F_0
\]
define
\[
 X=X(r,s,t,u)=
 \begin{pmatrix}
 0&r&s&0\\
 t&0&0&-r\\
 u&0&0&-s/q\\
 0&-t&-qu&0
 \end{pmatrix},
 \qquad
 \Delta_X=qru-st.
\]
A direct calculation gives
\[
 X^*H_q+H_qX=0.
\]
For $\varepsilon\in F_0$ sufficiently small, $I-\varepsilon X$ is invertible and the Cayley transform
\[
 C_\varepsilon=(I+\varepsilon X)(I-\varepsilon X)^{-1}
\]
lies in $\SU(H_q)(F_0)$.  Put
\[
 B_\varepsilon=C_\varepsilon AC_\varepsilon^{-1},
 \qquad
 G_\varepsilon=AB_\varepsilon.
\]
At $\varepsilon=0$,
\[
 \chi_{G_0}(Z)=(Z-a^2)(Z-a^{-2})(Z-1)^2.
\]
Hensel factorization therefore gives, for sufficiently small $\varepsilon$, a unique factorization
\[
 \chi_{G_\varepsilon}(Z)=p_\varepsilon(Z)q_\varepsilon(Z),
 \qquad
 \deg p_\varepsilon=\deg q_\varepsilon=2,
\]
with $q_0=(Z-1)^2$; the second factor is the local rotational factor.

\begin{proposition}\label{prop:bending}
With the notation above,
\[
 \disc(q_\varepsilon)
 =-
 \frac{64(a-1)^2}{q(a+1)^2}
 \Delta_X^2\varepsilon^4(1+O(\varepsilon)).
\]
In particular, if
\[
 -q^{-1}\notin F^{\times2},
 \qquad \Delta_X\ne0,
\]
then every sufficiently small nonzero bending has nonzero transverse discriminant and produces
\[
 \delta_{\rm rot}(G_\varepsilon)=[-q^{-1}]\ne1,
\]
so $q_{G_\varepsilon}$ is irreducible over $F$.
If $\Delta_X=0$, the bending preserves a nonzero line in the anisotropic plane and the generated group remains in a proper algebraic subgroup.
\end{proposition}

\begin{proof}
Decompose the space as
\[
 E\oplus W=\langle e_1,e_4\rangle\oplus\langle e_2,e_3\rangle.
\]
The Cayley transform has expansion
\[
 C_\varepsilon=I+2\varepsilon X+2\varepsilon^2X^2+O(\varepsilon^3).
\]
Since $A$ acts trivially on $W$, the first-order contribution to the $W\to W$ block of $G_\varepsilon$ vanishes.  The Schur complement at the double eigenvalue $1$ gives the effective second-order operator
\[
 K_X=
 \frac{4(a-1)\Delta_X}{a+1}
 \begin{pmatrix}0&1\\-1/q&0\end{pmatrix}.
\]
Hence
\[
 \disc(K_X)=-\frac{64(a-1)^2}{q(a+1)^2}\Delta_X^2.
\]
The two roots of $q_\varepsilon$ have expansions
\[
 1+\varepsilon^2\kappa_1+O(\varepsilon^3),
 \qquad
 1+\varepsilon^2\kappa_2+O(\varepsilon^3),
\]
where $\kappa_1,\kappa_2$ are the eigenvalues of $K_X$, and the stated discriminant formula follows.  Since the subgroup of squares is open in $F^\times$, the factor $1+O(\varepsilon)$ is a square for all sufficiently small $\varepsilon$.

Finally, the $W\to E$ block of $X$ is
\[
 \begin{pmatrix}r&s\\-t&-qu\end{pmatrix}
\]
with determinant $-\Delta_X$.  If $\Delta_X=0$, it has a nonzero kernel $L\subset W$; the $W\to W$ block of $X$ vanishes, so $X$ fixes $L$ pointwise.  The same line is fixed by $C_\varepsilon$, $A$, and $B_\varepsilon$.
\end{proof}

\begin{corollary}\label{cor:bending-Qi3-en}
Let
\[
 H_3=\mathbb H\oplus\langle1,3\rangle,
 \qquad A=\diag(3,1,1,1/3).
\]
At the two obstructive primes $p=2,3$ one has
\[
 -1/3\notin\Q_p(i)^{\times2}.
\]
Consequently every direction above with $3ru-st\ne0$ produces, for sufficiently small nonzero $\varepsilon\in\Q_p$, a word $AB_\varepsilon$ with irreducible rotational factor and local spectral degree two.
\end{corollary}

\begin{proof}
For $p=3$, the valuation of $-1/3$ in the unramified quadratic field $\Q_3(i)$ is odd, so it cannot be a square.  For $p=2$, if $-1/3$ were a square in $\Q_2(i)$, then $3$ would also be a square because $-1=i^2$; this would force $\Q_2(\sqrt3)=\Q_2(i)$.  Equality of these quadratic extensions would imply that the ratio of their defining square classes, $-3$, is a square in $\Q_2$, but $-3\equiv5\pmod8$ is not a $2$-adic square.  Proposition~\ref{prop:bending} now applies.
\end{proof}

The same expansion admits genuinely complex directions.  For the Gaussian model set
\[
 X_*=
 \begin{pmatrix}
 0&1&1&0\\
 1&0&0&-1\\
 1+i&0&0&-1/3\\
 0&-1&-3+3i&0
 \end{pmatrix}.
\]
Then
\[
 X_*^*H_3+H_3X_*=0.
\]
If
\[
 C_\varepsilon=(I+\varepsilon X_*)(I-\varepsilon X_*)^{-1},
 \qquad B_\varepsilon=C_\varepsilon AC_\varepsilon^{-1},
\]
the effective matrix is
\[
 K_*=
 \begin{pmatrix}
 0&4-6i\\
 -4/3-2i&-4i
 \end{pmatrix},
 \qquad
 \disc(K_*)=-\frac{256}{3}=-3\left(\frac{16}{3}\right)^2.
\]
Thus, for rational $\varepsilon$ sufficiently small in the corresponding $2$-adic and $3$-adic topologies, the rotational factor of $AB_\varepsilon$ is non-split at both obstructive places.

\begin{proposition}\label{prop:zariski-bending-en}
For the direction $X_*$ above, there is a finite algebraic exceptional set $\mathcal E\subset\overline{\Q}$ such that, for every $\varepsilon\in\Q\setminus\mathcal E$ for which $I-\varepsilon X_*$ is invertible, the Cayley transform lies in $\SU(H_3,\Q(i))$ and
\[
 \Gamma_\varepsilon=\langle A,B_\varepsilon\rangle<\SU(H_3,\Q(i))
\]
is Zariski-dense in the complex form $\SL_4$.  Consequently there are rational parameters arbitrarily small simultaneously in $\Q_2$ and $\Q_3$ for which $\Gamma_\varepsilon$ is Zariski-dense and $AB_\varepsilon$ has local spectral degree two at both obstructive places.
\end{proposition}

\begin{proof}
The Zariski closure of $\langle A\rangle$ contains the torus
\[
 z\longmapsto\diag(z,1,1,z^{-1}),
\]
so its Lie algebra contains
\[
 H_0=\diag(1,0,0,-1).
\]
The closure of $\langle B_\varepsilon\rangle$ contains the conjugate $C_\varepsilon H_0C_\varepsilon^{-1}$.  At $\varepsilon=1$, reduce modulo $5$ using $i\mapsto2$.  Put $Z_1=H_0$, $Z_2=C_1H_0C_1^{-1}$ and define recursively
\[
\begin{gathered}
 Z_3=[Z_1,Z_2],\quad Z_4=[Z_1,Z_3],\quad Z_5=[Z_2,Z_3],\\
 Z_6=[Z_1,Z_4],\quad Z_7=[Z_1,Z_5],\quad Z_8=[Z_2,Z_5],\\
 Z_9=[Z_3,Z_4],\quad Z_{10}=[Z_3,Z_5],\quad Z_{11}=[Z_4,Z_5],\\
 Z_{12}=[Z_1,Z_6],\quad Z_{13}=[Z_1,Z_7],\quad
 Z_{14}=[Z_1,Z_8],\quad Z_{15}=[Z_1,Z_9].
\end{gathered}
\]
In the coordinate system consisting of the twelve off-diagonal entries followed by the first three diagonal entries, the $15\times15$ matrix with columns $Z_1,\ldots,Z_{15}$ has determinant
\[
 2\pmod5.
\]
Thus these Lie words span $\mathfrak{sl}_4(\mathbb F_5)$.  The same determinant before reduction is a rational function of $\varepsilon$ that is not identically zero, so it can vanish at only finitely many algebraic values (away from poles).  This proves Zariski density outside a finite exceptional set.  The companion script \texttt{CERTIFICADO\_ZARISKI\_BENDING\_v1.py}, archived in the reproducibility package \cite{SplittingFieldsZenodo2026}, reproduces $C_1$, the fifteen Lie words, and the determinant $2\pmod5$.

Weak approximation supplies infinitely many $\varepsilon\in\Q$ that are simultaneously as small as desired at $2$ and $3$ and avoid the finite exceptional set.  The preceding local discriminant calculation then gives degree two at both places.
\end{proof}

The bending family gives an explicit model in which odd local spectra are destroyed by transverse Zariski-dense perturbations of the rank-one example.

\section{Derived Levi holonomy and asymptotic composition}\label{sec:reset}

The bending calculation points to the invariant behind the phenomenon.  In rank one the relevant rotational information lives in the derived anisotropic Levi factor, while the usual spectral data of large semigroups is obtained only after an abelian projection.  Let $P=MAN^+$ be a minimal parabolic.  The key distinction is that the usual Guivarc'h spectrum is obtained after projection to an abelian quotient such as
\[
 MA/[M,M],
\]
whereas the rotational polynomial lives in the derived anisotropic Levi factor.  Therefore the projected spectrum is intrinsically blind to the split/non-split rotational datum and, on the separable locus, to $\delta_{\rm rot}$; compare \cite{GuivarchRaugi2007,Guivarch2008}.

\begin{proposition}\label{prop:abelian-blind}
Let
\[
 F_0=\Q_3,
 \qquad F=\Q_3(i),
 \qquad W=\langle1,3\rangle.
\]
There are regular elements $m_{\rm sp},m_{\rm ns}\in\SU(W)(F_0)$ with the same image in every quotient that kills the derived anisotropic Levi factor but with
\[
 \delta(m_{\rm sp})=1,
 \qquad
 \delta(m_{\rm ns})=[3]\ne1
 \quad\text{in }F^\times/F^{\times2}.
\]
In particular, neither the split Jordan projection nor the abelianized Levi spectrum determines the rotational splitting status; on the separable locus, neither determines the rotational discriminant square class.
\end{proposition}

\begin{proof}
Take
\[
 m_{\rm sp}=\begin{pmatrix}i&0\\0&-i\end{pmatrix}.
\]
Its characteristic polynomial is $X^2+1$ and its discriminant is $-4=(2i)^2$, a square in $F$.  Since $F/F_0$ is unramified, the norm map on units is surjective; choose $b\in F^\times$ with $b\bar b=-1$ and define
\[
 m_{\rm ns}=\begin{pmatrix}2&-3\bar b\\ b&2\end{pmatrix}.
\]
A direct calculation gives
\[
 m_{\rm ns}^*\begin{pmatrix}1&0\\0&3\end{pmatrix}m_{\rm ns}
 =\begin{pmatrix}1&0\\0&3\end{pmatrix},
 \qquad \det(m_{\rm ns})=1.
\]
Its characteristic polynomial is $X^2-4X+1$, with discriminant $12=4\cdot3$.  The class of $3$ is nontrivial in $F^\times/F^{\times2}$.  Both elements lie entirely in the derived anisotropic factor, hence both map trivially to every quotient that kills that factor, while their rotational square classes differ.
\end{proof}

The correct object is the full Bruhat cocycle before abelianization.  Fix a maximal split torus $S$, let $L=Z_G(S)$, and let $M=L/S$.  In relative rank one, $M$ is anisotropic and $M^{\rm der}$ is isogenous to the special unitary group of the transverse anisotropic plane.

\begin{lemma}\label{lem:levi-rotation}
Let $g$ be loxodromic and assume that its attracting point $x_+(g)$ lies in the chosen big Bruhat cell $N^-x_0$.  Then the Levi component of the Bruhat cocycle at $x_+(g)$ acts on
\[
 x_+(g)^\perp/x_+(g)
\]
with characteristic polynomial equal to the transverse factor $q_g$.  Hence its Levi conjugacy class determines whether $q_g$ is split or irreducible; when $q_g$ is separable, it also determines $\delta_{\rm rot}(g)$.
\end{lemma}

\begin{proof}
Write the attracting point uniquely as
\[
 x_+(g)=n^-x_0,
 \qquad n^-\in N^-(F_0).
\]
Since $g$ fixes $x_+(g)$,
\[
 p_g=(n^-)^{-1}gn^-\in P(F_0).
\]
By definition, the projection of $p_g$ to $L/S=M$ is the Bruhat cocycle value at $x_+(g)$.

The unipotent radical of $P$ acts trivially on the transverse quotient $x_0^\perp/x_0$, and the split torus $S$ also acts trivially there.  Thus the induced action of $p_g$ on this quotient factors canonically through $M$.  Transporting by $n^-$ identifies it with the action of $g$ on
\[
 x_+(g)^\perp/x_+(g).
\]
In the rank-one decomposition
\[
 V=\ell_+(g)\oplus W_g\oplus\ell_-(g),
\]
the natural map $W_g\to\ell_+(g)^\perp/\ell_+(g)$ is a Hermitian isomorphism intertwining the transverse block $R_g$.  Therefore the Levi cocycle has transverse characteristic polynomial $q_g$, hence the same split/non-split status and, on the separable locus, the same discriminant square class.
\end{proof}

\begin{lemma}\label{lem:uniform-reset}
Keep the rank-one notation above and let $t$ be loxodromic with attracting endpoint $x_0$ and $t\in Z_G(S)(F_0)$.  Let $s_1,\ldots,s_r\in G(F_0)$ be such that each $s_jx_0$ lies in the big Bruhat cell $N^-x_0$.  Then there is a compact open neighborhood $U$ of $x_0$ in the big cell with the following property.  For every $j$,
\[
 \sup_{y\in U}d(t^ns_jy,x_0)\longrightarrow0
 \qquad(n\to\infty),
\]
and, after choosing analytic coordinates on $U$, there are constants $C_j>0$ and $0<\rho<1$ such that
\[
 \operatorname{Lip}(t^ns_j|_U)\le C_j\rho^n
\]
for all sufficiently large $n$.  In particular, for all sufficiently large $n$,
\[
 t^ns_j(U)\subset U.
\]
\end{lemma}

\begin{proof}
Use the analytic chart
\[
 \phi:N^-(F_0)\longrightarrow N^-x_0,\qquad u\longmapsto ux_0.
\]
Because $x_0$ is the attracting endpoint of $t$, the automorphism
\[
 \alpha(u)=tut^{-1}
\]
of $N^-(F_0)$ is contracting.  Equivalently, on every negative-root space the eigenvalues of $\Ad(t)$ have absolute value $<1$.  Hence there is a compact open subgroup $V\subset N^-(F_0)$ and a compatible analytic ultrametric norm for which
\[
 \alpha(V)\subset V,
 \qquad
 \operatorname{Lip}(\alpha^n|_V)\le C\rho^n
\]
with $0<\rho<1$; moreover $\alpha^n(K)\to\{e\}$ uniformly for every compact $K\subset N^-(F_0)$.

Since each $s_jx_0$ belongs to the big cell, after shrinking $V$ we may assume that $s_j\phi(V)$ is contained in the big cell for every $j$.  Put $U=\phi(V)$ and let
\[
 K_j=\phi^{-1}(s_jU)\subset N^-(F_0).
\]
Each $K_j$ is compact.  For $y=ux_0\in s_jU$ one has
\[
 t^ny=(t^nut^{-n})x_0=\phi(\alpha^n(u)).
\]
Uniform contraction of $\alpha^n$ on $K_j$ therefore gives
\[
 \sup_{y\in U}d(t^ns_jy,x_0)\to0.
\]
On the compact analytic sets involved, the coordinate expression of $s_j$ has bounded derivative, while the derivative of $\alpha^n$ is $O(\rho^n)$ on the negative-root filtration.  The chain rule gives the asserted Lipschitz estimate.  The uniform convergence to $x_0$ finally implies $t^ns_j(U)\subset U$ for all large $n$.
\end{proof}

\begin{lemma}\label{lem:reset}
Let $\Lambda<G$ be discrete and torsion-free and let $t\in\Lambda$ be loxodromic.  Conjugate the local data so that the maximal split torus $S$, the minimal parabolic $P=MAN^+$, and the opposite unipotent group $N^-$ are adapted to the axis of $t$: the attracting endpoint is the base point $x_0\in G/P$ and $t\in Z_G(S)(F_0)$.  For elements $s_1,\ldots,s_r\in\Lambda$ whose Bruhat cocycle values
\[
 c(s_j)\in M(F_0)
\]
are defined at $x_0$, there exist sequences
\[
 n_1(k),\ldots,n_r(k)\to+\infty
\]
such that
\[
 w_k=t^{n_r(k)}s_r\cdots t^{n_1(k)}s_1
\]
is loxodromic for large $k$ and its compact Levi holonomy satisfies
\[
 [m_{w_k}]\longrightarrow[c(s_r)\cdots c(s_1)]
\]
in the conjugacy-class quotient of $M$.
\end{lemma}

\begin{proof}
Let $\bar t$ be the image of $t$ in the compact anisotropic quotient $M(F_0)$.  Compactness gives an unbounded sequence of return times $N_\nu$ such that
\[
 \bar t^{N_\nu}\longrightarrow e.
\]
Apply Lemma~\ref{lem:uniform-reset} to $s_1,\ldots,s_r$ and choose the corresponding compact open set $U$.  For each $k$ choose, successively and independently, return times
\[
 n_j(k)\in\{N_\nu\}_{\nu\ge1}
\]
so large that
\[
 \bar t^{n_j(k)}\to e,
 \qquad
 t^{n_j(k)}s_j(U)\subset U,
\]
and
\[
 \operatorname{Lip}(t^{n_j(k)}s_j|_U)<1/k,
 \qquad
 \sup_{y\in U}d(t^{n_j(k)}s_jy,x_0)<1/k.
\]
Set
\[
 B_{j,k}=t^{n_j(k)}s_j,
 \qquad
 w_k=B_{r,k}\cdots B_{1,k}.
\]
Then $w_k$ maps the complete compact ultrametric space $U$ into itself and is a strict contraction.  It therefore has a unique fixed point $x_k\in U$.  Since the last block resets all of $U$ into the $1/k$-ball about $x_0$,
\[
 d(x_k,x_0)=d(w_kx_k,x_0)<1/k,
\]
so $x_k\to x_0$.

The element $w_k$ is nontrivial for large $k$: the identity map has Lipschitz constant $1$ on the non-singleton set $U$, whereas $w_k$ has Lipschitz constant $<1$.  An element of a rank-one Bruhat--Tits tree is elliptic or loxodromic.  If $w_k$ were elliptic, it would lie in a compact vertex stabilizer; the intersection of such a stabilizer with the discrete group $\Lambda$ is finite, and torsion-freeness would force $w_k=e$.  Hence $w_k$ is loxodromic, with attracting endpoint $x_k$.

For the cocycle calculation put
\[
 y_{0,k}=x_k,
 \qquad
 z_{j,k}=s_jy_{j-1,k},
 \qquad
 y_{j,k}=t^{n_j(k)}z_{j,k}
 \quad(1\le j\le r).
\]
All $y_{j,k}$ lie in $U$, and Lemma~\ref{lem:uniform-reset} gives, uniformly in the preceding point,
\[
 y_{j,k}\longrightarrow x_0
 \qquad(0\le j\le r).
\]
Hence continuity of the Bruhat cocycle yields
\[
 \bar\gamma(s_j,y_{j-1,k})\longrightarrow c(s_j).
\]
If $z=n^-x_0$ lies in the big cell, then $t\in Z_G(S)$ normalizes $N^-$ and fixes $x_0$, so uniqueness of Bruhat coordinates gives the exact identity
\[
 \bar\gamma(t^n,z)=\bar t^{\,n}.
\]
Therefore
\[
 \bar\gamma(t^{n_j(k)},z_{j,k})
 =\bar t^{n_j(k)}\longrightarrow e.
\]
Applying the cocycle identity successively to
\[
 w_k=t^{n_r(k)}s_r\cdots t^{n_1(k)}s_1
\]
gives
\[
 \bar\gamma(w_k,x_k)
 \longrightarrow
 c(s_r)\cdots c(s_1).
\]
By Lemma~\ref{lem:levi-rotation}, the conjugacy class of this cocycle value is exactly the compact rotational holonomy of $w_k$, proving the desired limit.
\end{proof}

The one-block precursor $t^ns$, after projection to $MA/[M,M]$, already appears in Guivarc'h--Raugi.  Retaining the non-abelian anisotropic quotient allows the construction to be iterated over finitely many blocks.

The reset lemma solves the realization problem for any prescribed finite product of cocycle values.  The next step is algebraic: show that these values generate enough of the Levi factor, and then convert Zariski density into local openness.

\begin{proposition}\label{prop:large-holonomy}
Assume in addition that $\Lambda$ is Zariski-dense.  Let $C_t(\Lambda)$ be the closed subgroup of $M(F_0)$ generated by the cocycle values $c(s)$.  Then $C_t(\Lambda)$ is Zariski-dense in $M$, and
\[
 D_t(\Lambda)=\overline{[C_t(\Lambda),C_t(\Lambda)]}
\]
is Zariski-dense in $M^{\rm der}$.  Moreover the conjugacy classes of elements of $C_t(\Lambda)$ lie in the closure of rotational holonomies realized by actual loxodromic words of $\Lambda$.
\end{proposition}

\begin{proof}
The intersection of $\Lambda$ with the big Bruhat cell $N^-LN^+$ is Zariski-dense in that open cell.  Projection
\[
 N^-LN^+\longrightarrow M,
 \qquad n^-\ell n^+\longmapsto \ell\bmod S,
\]
is a dominant algebraic map.  Hence the cocycle values $c(s)$ are Zariski-dense in $M$, and so is the closed subgroup $C_t(\Lambda)$ that they generate.  The Zariski closure of its commutator subgroup is therefore $M^{\rm der}$.

Because $M(F_0)$ is compact, the closure $P$ of the positive semigroup generated by the $c(s)$ is a group.  Indeed, fix $x\in P$.  The closure of the positive powers of $x$ is compact.  Given any identity neighborhood $U$, choose a smaller symmetric neighborhood $V$ with $V^{-1}V\subset U$ and cover the compact cyclic closure by finitely many left translates of $V$.  Two sufficiently far positive powers $x^a,x^b$ with $a>b$ lie in the same translate, whence $x^{a-b}\in V^{-1}V\subset U$.  Applying this to a decreasing neighborhood basis at $e$ yields positive return times $n_j\to\infty$ such that
\[
 x^{n_j}\longrightarrow e.
\]
Consequently
\[
 x^{n_j-1}\longrightarrow x^{-1},
\]
so $x^{-1}\in P$.  Thus $P$ contains inverses and is a closed subgroup.  Since it contains all generators $c(s)$, it equals $C_t(\Lambda)$.  Finite positive products are therefore dense in $C_t(\Lambda)$.  Lemma~\ref{lem:reset} realizes every such product as a limit of holonomies of loxodromic words.  Passing to conjugacy classes and then to closure proves the last assertion.
\end{proof}

Thus finite products of cocycle values are not merely abstract Levi data: they are limits of rotational holonomies of actual loxodromic words.  It remains to show that the derived part of this holonomy is locally large.

\begin{lemma}\label{lem:derived-open}
If $F_0=\Q_p$, then $D_t(\Lambda)$ is open in $M^{\rm der}(\Q_p)$.
\end{lemma}

\begin{proof}
The subgroup $D_t(\Lambda)$ is closed and $p$-adic analytic.  Let
\[
 \mathfrak d=\operatorname{Lie}(D_t(\Lambda)),
 \qquad
 \mathfrak m=\operatorname{Lie}(M^{\rm der})(\Q_p).
\]
The Lie algebra $\mathfrak m$ is simple.  If $\mathfrak d=0$, then $D_t(\Lambda)$ is discrete; compactness of $M^{\rm der}(\Q_p)$ would make it finite, contradicting its Zariski density.  Hence $\mathfrak d\ne0$.

The space $\mathfrak d$ is stable under $\Ad(D_t(\Lambda))$.  Its stabilizer in the adjoint representation is algebraic over the same base field $\Q_p$.  Zariski density forces that stabilizer to be all of $M^{\rm der}$, so $\mathfrak d$ is a nonzero ideal of the simple algebra $\mathfrak m$.  Therefore $\mathfrak d=\mathfrak m$.  A closed subgroup with full Lie algebra is open.  This is the precise special case needed here of the broader compact-subgroup openness technology developed by Pink \cite{Pink1998}.
\end{proof}

\begin{lemma}\label{lem:derived-transverse}
Let $W_0=x_0^\perp/x_0$ be the transverse anisotropic Hermitian plane at the base point.  The natural transverse action induces an $F_0$-isogeny
\[
 \tau:M^{\rm der}\longrightarrow \SU(W_0).
\]
If $D_t(\Lambda)$ is open in $M^{\rm der}(F_0)$, then $\tau(D_t(\Lambda))$ contains an identity neighborhood in $\SU(W_0)(F_0)$.
\end{lemma}

\begin{proof}
The derived Levi acts nontrivially on the two-dimensional transverse quotient, and its image is the derived group of the corresponding unitary action, namely $\SU(W_0)$.  Both connected groups are of type $A_1$, so the induced morphism is a central isogeny.  In characteristic zero its differential is an isomorphism of Lie algebras.  The $F_0$-analytic inverse-function theorem therefore makes $\tau$ an open map on a sufficiently small identity neighborhood.  Since $D_t(\Lambda)$ is open, its image contains an identity neighborhood in $\SU(W_0)(F_0)$.
\end{proof}

\begin{lemma}\label{lem:accumulate}
Let $W$ be an anisotropic Hermitian plane over $F/F_0$.  In $\SU(W)(F_0)$ there are regular elements $u_j\to e$ whose characteristic-polynomial discriminants are non-squares in $F$.
\end{lemma}

\begin{proof}
The anisotropic group $\SU(W)$ is an inner form of $\SL_2$ and may be realized as the reduced-norm-one group of a quaternion division algebra $D/F_0$.  Choose a quadratic extension $E/F_0$ with $E\ne F$.  Since the Brauer invariant of $D$ is $1/2$, every quadratic extension splits $D$, and the local embedding criterion realizes $E$ as a maximal subfield.  Thus the norm-one torus
\[
 T=R^1_{E/F_0}\mathbf G_m
\]
embeds in $\SU(W)$.

Choose a regular infinite-order element $u\in T(F_0)$.  Compactness gives powers $u^{n_j}\to e$.  Its two eigenvalues over $F$ lie in $EF$ and are interchanged by the nontrivial automorphism of $EF/F$.  If a positive power had both eigenvalues in $F$, that automorphism would give $\lambda^n=\lambda^{-n}$ and hence $\lambda^{2n}=1$, contradicting infinite order.  Thus every $u^{n_j}$ has nonsplit quadratic characteristic polynomial over $F$, while tending to the identity.
\end{proof}

The two ingredients now meet.  Derived holonomy contains an identity neighborhood in the transverse anisotropic unitary group, and nonsplit regular elements accumulate at the identity there.  The realization result transfers one of these elements back to an actual word of $\Lambda$.

\begin{theorem}\label{thm:local-rank1}
Let $F/\Q_p$ be quadratic and let $V$ be a four-dimensional non-quasi-split Hermitian space.  There is no subgroup
\[
 \Lambda<\SU(V)(\Q_p)
\]
which is simultaneously discrete, torsion-free, Zariski-dense, and for which the transverse factor $q_g$ splits over $F$ for every $1\ne g\in\Lambda$.  Equivalently, every such subgroup contains a loxodromic element $g$ for which $q_g$ is irreducible over $F$, and hence whose local spectral field over $F$ has degree two.
\end{theorem}

\begin{proof}
Choose a reset element $t$.  By Proposition~\ref{prop:large-holonomy} and Lemma~\ref{lem:derived-open}, the closed derived-holonomy group $D_t(\Lambda)$ is an open subgroup of $M^{\rm der}(\Q_p)$.  Lemma~\ref{lem:derived-transverse} shows that its transverse image contains an identity neighborhood in the anisotropic group $\SU(W_0)(\Q_p)$.  By Lemma~\ref{lem:accumulate}, that neighborhood contains a regular element with irreducible quadratic characteristic polynomial over $F$.  Choose $d\in D_t(\Lambda)\subset C_t(\Lambda)$ whose transverse image has this property.

Proposition~\ref{prop:large-holonomy} places the conjugacy class of $d$ in the closure of rotational holonomies of actual loxodromic words of $\Lambda$.  Let
\[
 \mathcal N\subset M(\Q_p)
\]
be the set of elements whose transverse characteristic polynomial is regular and irreducible over $F$.  This set is open: regularity is the nonvanishing of the discriminant, and in a non-archimedean local field the subgroup $F^{\times2}\subset F^\times$ is open, so the square class of a nonzero discriminant is locally constant.  The set $\mathcal N$ is also invariant under conjugation.

Let
\[
 q_M:M(\Q_p)\longrightarrow M(\Q_p)/\!\operatorname{Conj}
\]
be the topological conjugacy-class quotient.  Conjugation acts by homeomorphisms, so $q_M$ is an open quotient map; since $\mathcal N$ is conjugacy-saturated, $q_M(\mathcal N)$ is an open neighborhood of $[d]$.  Because $[d]$ lies in the closure of the realized rotational classes, one of those classes lies in $q_M(\mathcal N)$.  The corresponding loxodromic word therefore has irreducible transverse factor.  Corollary~\ref{cor:rank1-parity} then gives local spectral degree two.
\end{proof}

\section{Local--global descent in dimension four}\label{sec:assembly}

The local analysis has a simple global consequence.  Once one can force an even spectral extension at each obstructive place, the Hermitian norm class disappears after passing to a finite compositum.

\begin{lemma}\label{lem:compositum}
Let $K/\Q$ be imaginary quadratic and let $\Delta\in\Q_{>0}$.  For each $p\in S_\Delta$, suppose there is a finite Galois extension $L_p/K$, stable under complex conjugation, such that every completion of $L_p$ over $p$ has even degree over $K_p$.  Let $L$ be the compositum of the $L_p$ and let $L_0=L^{\bar{\ }}$.  Then
\[
 \Delta\in N_{L/L_0}(L^\times).
\]
\end{lemma}

\begin{proof}
Every $p\in S_\Delta$ is non-split in $K$, hence $K$ has a unique place above $p$.  Since $L_p/K$ is Galois, all local degrees above that place are equal and even.  The compositum $L/K$ remains Galois and stable under complex conjugation.  Put $L_0=L^{\bar{\ }}$, so $L/L_0$ is quadratic.

Let $v$ be a place of $L_0$ over $p\in S_\Delta$.  If $L/L_0$ splits at $v$, the local norm is surjective.  Otherwise let $w$ be the unique place of $L$ over $v$.  The completion $L_w$ contains a completion of $L_p$, so
\[
 [L_w:K_p]
\]
is even.  Conjugation acts nontrivially on $K_p$ and fixes $L_{0,v}$, hence
\[
 K_p\cap L_{0,v}=\Q_p.
\]
Since both $K_p/\Q_p$ and $L_w/L_{0,v}$ have degree two and $K_p\subset L_w$, it follows that
\[
 L_w=K_pL_{0,v}.
\]
Comparing the two towers therefore gives
\[
 [L_{0,v}:\Q_p]=[L_w:K_p],
\]
which is even.  Restriction of the local Brauer class multiplies the invariant $1/2$ by this even degree, so the class vanishes over $L_{0,v}$.  Equivalently, the corresponding cyclic algebra for the quadratic extension $L_w/L_{0,v}$ splits, and $\Delta$ is a local norm from $L_w/L_{0,v}$.

At places outside $S_\Delta$ the class was already trivial.  At the archimedean places, either the extension splits or the positivity of $\Delta$ makes it a norm from $\C^\times$ to $\R_{>0}$.  The Hasse norm theorem for the cyclic quadratic extension $L/L_0$ now gives
\[
 \Delta\in N_{L/L_0}(L^\times).
\]
\end{proof}

The preceding results give the following theorem.

\begin{theorem}\label{thm:partial-total}
Let $K/\Q$ be imaginary quadratic,
\[
 H_K=\mathbb H\oplus\langle1,q\rangle,
 \qquad q\in\Q_{>0},
\]
and let $\Gamma<\SU(H_K,K)$ be finitely generated and Zariski-dense.  Assume that for every obstructive prime $p\in S_q$, at least one of the following holds:
\begin{enumerate}[label=(\roman*)]
\item the closure of $\Gamma$ in $\SU(H_K)(\Q_p)$ is non-discrete;
\item the local Hermitian space $H_K\otimes\Q_p$ is non-quasi-split and the local closure is discrete.
\end{enumerate}
Then there exists a finite conjugation-stable extension
\[
 L\subset F_{\rm eig}(\Gamma)
\]
such that
\[
 H_K\sim_LJ.
\]
In particular $F_{\rm eig}(\Gamma)$ is a unitary field of definition.
\end{theorem}

\begin{proof}
If $S_q=\varnothing$, the finite obstruction is already trivial and there is nothing to prove.  Otherwise pass to a torsion-free finite-index subgroup $\Gamma'\subset\Gamma$; it remains Zariski-dense, and $F_{\rm eig}(\Gamma')\subset F_{\rm eig}(\Gamma)$.

Fix $p\in S_q$.  In case (i), let
\[
 H_p=\overline{\Gamma'}\subset\SU(H_K)(\Q_p).
\]
Lemma~\ref{lem:closure-dichotomy} makes $H_p$ open.  Choose an identity neighborhood $\Omega_p\subset H_p$ and apply Lemma~\ref{lem:biquad} inside $\Omega_p$ to obtain an open set $U_p\subset H_p$ on which every characteristic polynomial has local splitting field containing a quadratic extension of $K_p$.  Since $\Gamma'$ is dense in $H_p$, choose
\[
 \gamma_p\in\Gamma'\cap U_p.
\]
Then $K_p(E(\gamma_p))/K_p$ has even degree.  In case (ii), Theorem~\ref{thm:local-rank1} produces $\gamma_p\in\Gamma'$ with local splitting degree two.

In either case put
\[
 L_p=K(E(\gamma_p)).
\]
This is the splitting field over $K$ of the characteristic polynomial of $\gamma_p$, hence $L_p/K$ is finite Galois.  Unitarity makes the set of roots stable under $\lambda\mapsto\bar\lambda^{-1}$, so $L_p$ is stable under complex conjugation.  For one completion of $L_p$ above the unique place of $K$ over $p$, the local field is exactly the local splitting field $K_p(E(\gamma_p))$ and therefore has even degree over $K_p$.  Because $L_p/K$ is Galois, all completions above that place have the same local degree.  Thus $L_p$ satisfies the hypotheses of Lemma~\ref{lem:compositum}.

Taking the compositum over the finite set $S_q$ and applying that lemma kills the norm class $q$.  The resulting compositum lies in $F_{\rm eig}(\Gamma')\subset F_{\rm eig}(\Gamma)$, so it is a unitary field of definition for the original group $\Gamma$ as well.
\end{proof}

It remains to treat obstructive primes for which the local closure is discrete and the local Hermitian form is quasi-split.  Such places can occur because the obstructive class is defined relative to the chosen global standard form and need not coincide with the local non-quasi-split class.

\subsection{The quasi-split branch}

In the quasi-split four-dimensional case the local group has relative rank two.  The building is two-dimensional, not a tree.  There is no canonical decomposition
\[
 1+2+1
\]
with a unique anisotropic Hermitian plane, no unique quadratic rotational factor, and no single square-class discriminant that detects all even spectral extensions.  Consequently the proof of Theorem~\ref{thm:local-rank1} cannot simply be repeated.

A higher-rank solution should provide a replacement for at least one of the following rank-one mechanisms:
\begin{enumerate}[label=(\alph*)]
\item a canonical spectral factor whose non-splitting forces even local degree;
\item a multi-directional Levi holonomy whose open derived image meets an even spectral stratum;
\item an asymptotic composition theorem for words $g^nh^m$ or for several reset directions, strong enough to realize a prescribed higher-rank holonomy by actual group words;
\item a direct torus/generic-element argument showing that a discrete Zariski-dense subgroup cannot remain in the odd local spectral strata.
\end{enumerate}

Conjecture~\ref{conj:rank2} asks for such a higher-rank replacement.  A proof would upgrade Theorem~\ref{thm:partial-total} to Conjecture~\ref{conj:total-main}.

\section{Complex hyperbolic lattice examples}\label{sec:controls-en}

The preceding criteria become more transparent in familiar lattice examples.  Picard modular groups and their sisters are arithmetic and therefore fall under Theorem~\ref{thm:arith}.  Hybrid and totally loxodromic constructions show that Zariski-dense thin behavior is also abundant \cite{FalbelParker2006,FalbelFrancsicsParker2011,FalbelFrancsicsLaxParker2011,XieWangJiang2013,XieWangJiang2015,PaupertWells2020,FalbelPasquinelli2023,ParkerPlatis2008,GongopadhyayParsad2018}.  Non-arithmetic rank-four examples, in turn, show that the determinant obstruction and the total spectral field need not track arithmeticity itself.

Two such examples are especially informative.

\begin{proposition}\label{prop:G29-en}
Let $L=\Q(\zeta_{12})$ and $L_0=\Q(\sqrt3)$.  For the Hermitian form $H_{29}$ used by Deraux for the non-arithmetic lattice $\mathcal C(G_{29},3)$ \cite[(5.11), Proposition~5.4]{Deraux2020} one has
\[
 \det H_{29}=-\frac{2+\sqrt3}{9}.
\]
Since $\det J=-1$,
\[
 \frac{\det H_{29}}{\det J}
 =\frac{2+\sqrt3}{9}
 =N_{L/L_0}\left(\frac{1+\zeta_{12}}3\right).
\]
Hence the finite discriminant class relevant to the present unitary descent is already trivial over $L$.
\end{proposition}

\begin{proof}
In the coordinates of \cite[(5.11)]{Deraux2020},
\[
 \alpha=\frac1{z-1},
 \qquad
 \beta=\frac{1+i}{z(z-1)},
 \qquad z=e^{2\pi i/3},
\]
and
\[
 H_{29}=\begin{pmatrix}
 1&\alpha&0&0\\
 \bar\alpha&1&\alpha&\beta\\
 0&\bar\alpha&1&\alpha\\
 0&\bar\beta&\bar\alpha&1
 \end{pmatrix}.
\]
Substitution of $z=(-1+i\sqrt3)/2$ gives the displayed determinant.  Moreover
\[
 N_{L/L_0}(1+\zeta_{12})
 =(1+\zeta_{12})(1+\bar\zeta_{12})=2+\sqrt3.
\]
\end{proof}

Conceptually, this example separates two phenomena.  Its non-arithmeticity is caused by archimedean conjugates, whereas the finite norm class studied here is already trivial.

\begin{proposition}\label{prop:DM-en}
Let $H_{\rm DM}$ be the Hermitian form used in \cite[(6.6)]{Deraux2020} for the group $\mathcal C(B_4,3,4)$, identified there with the Deligne--Mostow lattice of weights $(3,3,3,3,5,7)/12$; see \cite[Proposition~6.1]{Deraux2020}.  Thus
\[
 H_{\rm DM}=\begin{pmatrix}
 0&0&0&1\\
 0&2&-1-i&0\\
 0&-1+i&2&0\\
 1&0&0&0
 \end{pmatrix}.
\]
Then
\[
 \det H_{\rm DM}=-2,
 \qquad
 \frac{\det H_{\rm DM}}{\det J}=2=N_{\Q(i)/\Q}(1+i).
\]
In fact, for
\[
 P=\begin{pmatrix}
 1&0&0&0\\
 0&(1+i)^{-1}&(1+i)/2&0\\
 0&0&1&0\\
 0&0&0&1
 \end{pmatrix}
\]
one has
\[
 P^*H_{\rm DM}P=J.
\]
\end{proposition}

\begin{proof}
The central Hermitian block
\[
 C=\begin{pmatrix}2&-1-i\\-1+i&2\end{pmatrix}
\]
has determinant $2$.  The vectors
\[
 u=\binom{(1+i)^{-1}}0,
 \qquad
 w=\binom{(1+i)/2}1
\]
satisfy $u^*Cu=w^*Cw=1$ and $u^*Cw=0$.  Replacing the central basis by $u,w$ gives the displayed matrix $P$ and the identity $P^*H_{\rm DM}P=J$.
\end{proof}

These non-arithmetic rank-four examples have trivial finite Hermitian obstruction of the type considered here, illustrating that non-arithmeticity and failure of unitary descent over a prescribed field are distinct phenomena.

\section{Traces, spectra, and fields of definition}\label{sec:trace-vs-eigen}

At this point the spectral field can be compared back with the trace field.  The distinction that organized the paper reappears in a particularly clean form: characteristic polynomials are controlled by traces, while their roots may still carry the arithmetic information needed for unitary descent.

The splitting-field problem therefore separates into two levels.  The first is exact: characteristic polynomials of group elements are determined by traces of words.  The second is genuinely arithmetic: can the Hermitian obstruction be removed without adjoining all of the roots of those polynomials?

For $g\in\SU(3,1)$ one has
\[
 \chi_g(X)=X^4-\tr(g)X^3
 +\frac{\tr(g)^2-\tr(g^2)}2X^2
 -\tr(g^{-1})X+1.
\]
Since $g^2,g^{-1}\in\Gamma$, all coefficients belong to the trace field
\[
 k_{\rm tr}=\Q(\tr\Gamma).
\]

\begin{proposition}\label{prop:trace-splitting-equivalence}
Let $F_{\rm spl}$ be the compositum over $k_{\rm tr}$ of the splitting fields of $\chi_g$ for all $g\in\Gamma$.  Then
\[
 F_{\rm spl}=\Q\bigl(E(g):g\in\Gamma\bigr).
\]
Consequently, after adjoining a chosen linear field of definition $K$,
\[
 KF_{\rm spl}=K\bigl(E(g):g\in\Gamma\bigr)=F_{\rm eig}(\Gamma).
\]
\end{proposition}

\begin{proof}
Each splitting field is generated by the eigenvalues of the corresponding element, so $F_{\rm spl}$ is contained in the field generated by all eigenvalues.  Conversely every eigenvalue of every $g$ is a root of $\chi_g$, hence belongs to its splitting field.  The final identity follows by adjoining $K$.
\end{proof}

Thus the splitting-field formulation is an intrinsic trace-theoretic presentation of the same total spectral extension once the linear field $K$ has been fixed.  The stronger trace-only form of the question remains open.

\begin{problem}\label{prob:trace-only}
Under what additional hypotheses does the trace field $k_{\rm tr}$ itself, or a finite extension determined functorially from trace data but smaller than the full spectral compositum, already kill the Hermitian similarity obstruction?  In particular, can one characterize when the auxiliary eigenvalue field of a distinguished regular loxodromic element is unnecessary?
\end{problem}

The results of this paper show why this question is subtle.  Traces determine the characteristic polynomials, but the Hermitian embedding class of a fixed torus may remain invisible to its abstract spectrum.  On the other hand, the arithmetic theorem and the local even-degree criterion show that spectra of suitable additional words do recover precisely the norm classes that are missing from the one-element field.

\section{The remaining quasi-split case}

The whole argument leaves one local configuration untouched: a quasi-split obstructive place of relative rank two.  This is not a technical variant of the rank-one case, because the distinguished transverse anisotropic plane disappears.  The preceding results therefore reduce the unresolved part of the $\SU(3,1)$ problem to this local question.  In relative rank one the transverse anisotropic plane gives a distinguished quadratic factor, and the full Bruhat cocycle recovers enough derived Levi holonomy to force a non-split rotation.  Neither feature has a direct analogue in relative rank two.

Two approaches seem natural.  One is to retain several Levi components simultaneously and prove an asymptotic composition law for words
\[
 g^{n_r}h_r\cdots g^{n_1}h_1,
\]
or for two-parameter words $g^nh^m$, with convergence in an appropriate higher-rank compact or derived quotient.  The other is to use generic maximal tori and show that a discrete Zariski-dense subgroup cannot have odd local splitting degree for every element at a fixed obstructive place.

Thus the field-of-definition problem has been reduced to a sharply localized higher-rank realization question.  The rank-one argument succeeds because a canonical quadratic transverse factor converts Levi holonomy into even local spectral degree; a solution in rank two must replace that factor by genuinely higher-rank spectral data.  Either of the two approaches above would supply the missing local input in Conjecture~\ref{conj:rank2} and complete the present strategy toward Conjecture~\ref{conj:total-main}.

\appendix
\section{A finite-field certificate for the bending computation}\label{app:bending-certificate}
For Proposition~\ref{prop:zariski-bending-en}, reduction at $\varepsilon=1$ modulo $5$ with $i\mapsto2$ gives
\[
 C_1=
 \begin{pmatrix}
 1&3&4&4\\
 1&2&1&0\\
 4&0&2&4\\
 1&2&3&3
 \end{pmatrix}
 \in\SL_4(\mathbb F_5).
\]
Starting from $Z_1=H_0$ and $Z_2=C_1H_0C_1^{-1}$ and using the recursive Lie words displayed in the proof of Proposition~\ref{prop:zariski-bending-en}, form the $15\times15$ coordinate matrix whose rows are the twelve off-diagonal coordinates followed by the first three diagonal coordinates.  Direct elimination over $\mathbb F_5$ gives
\[
 \det(Z_1,\ldots,Z_{15})=2\ne0\pmod5.
\]
This finite calculation is sufficient to certify the non-vanishing of the Lie-algebra minor used in the proof.  The companion script \texttt{CERTIFICADO\_ZARISKI\_BENDING\_v1.py}, archived with a permanent DOI in \cite{SplittingFieldsZenodo2026}, performs the same computation using only integer arithmetic modulo $5$.

\section{A finite-field certificate for the Schottky computation}\label{app:schottky-certificate}
For Proposition~\ref{prop:schottky}, reduce modulo $5$ with $i\mapsto2$.  Then
\[
 \bar A_s=\diag(3,2,3,2),
 \qquad
 \bar C_s=
 \begin{pmatrix}
 1&0&3&2\\
 3&2&0&4\\
 3&3&1&2\\
 1&1&0&4
 \end{pmatrix},
\]
and $\bar B_s=\bar C_s\bar A_s\bar C_s^{-1}$.  Starting with
\[
 \bar X=\diag(1,0,0,-1),
 \qquad
 \bar Y=\bar C_s\bar X\bar C_s^{-1},
\]
form the fifteen matrices $Z_1,\ldots,Z_{15}$ listed in the proof of Proposition~\ref{prop:schottky}.  In the same $15$ coordinates on $\mathfrak{sl}_4$ used above, the coordinate matrix with these columns satisfies
\[
 \det(Z_1,\ldots,Z_{15})=1\ne0\pmod5.
\]
Thus the Lie algebra generated by the indicated adjoint and commutator operations is all of $\mathfrak{sl}_4(\mathbb F_5)$, and the characteristic-zero determinant is nonzero.  The pure-Python script \path{CERTIFICADO_ZARISKI_SCHOTTKY_v1.py}, archived in the Zenodo reproducibility package \cite{SplittingFieldsZenodo2026}, performs this calculation exactly and terminates with \texttt{CERTIFICATE: PASS}.

\subsection*{Computational reproducibility}
Version 1.1 of the public Zenodo reproducibility package \cite{SplittingFieldsZenodo2026} archives both exact finite-field certificates used in this paper: \path{CERTIFICADO_ZARISKI_BENDING_v1.py} for Proposition~\ref{prop:zariski-bending-en} and \path{CERTIFICADO_ZARISKI_SCHOTTKY_v1.py} for Proposition~\ref{prop:schottky}.  The archive also contains execution instructions and a SHA-256 manifest.  Both certificates use only exact integer arithmetic modulo $5$; their successful outputs are respectively
\[
 \det=2\pmod5
 \qquad\text{and}\qquad
 \det=1\pmod5.
\]
No numerical approximation is used in either certificate.

\end{document}